\documentclass[9pt,reqno]{amsproc}
\pdfoutput=1

\usepackage{amsmath}
\usepackage{amsthm}
\usepackage{amssymb}

\usepackage{mathpazo} %
\usepackage[scaled]{helvet} %
\usepackage{courier} %
\normalfont
\usepackage[T1]{fontenc}

\usepackage{hyperref}
\usepackage{wasysym}

\usepackage[all]{xy}

\usepackage{graphicx}

\usepackage{mathtools}
\usepackage{calc}

\usepackage{wrapfig}

\newcommand{\blackboardbold}[1]{\ensuremath{\mathbf{#1}}}

\def\ZZ{\blackboardbold{Z}}
\def\QQ{\blackboardbold{Q}}
\def\RR{\blackboardbold{R}}
\def\CC{\blackboardbold{C}}

\def\GG{\blackboardbold{G}}

\newcommand{\mathtext}[1]{\ensuremath{\mathrm{#1}}}

\newcommand{\RP}{\ensuremath{\blackboardbold{RP}}}
\newcommand{\CP}{\ensuremath{\blackboardbold{CP}}}
\newcommand{\HP}{\ensuremath{\blackboardbold{HP}}}
\newcommand{\OP}{\ensuremath{\blackboardbold{CaP}}}
\newcommand{\MO}{\ensuremath{\mathtext{MO}}}

\newcommand{\BU}{\ensuremath{\mathtext{BU}}}
\newcommand{\ku}{\ensuremath{\mathtext{ku}}}
\newcommand{\BO}{\ensuremath{\mathtext{BO}}}
\newcommand{\BSO}{\ensuremath{\mathtext{BSO}}}
\newcommand{\SO}{\ensuremath{\mathtext{SO}}}
\newcommand{\BP}{\ensuremath{\mrm{BP}}}

\newcommand{\BBPb}[1]{\ensuremath{\mrm{\mbf{BP}}\langle\mbf{#1}\rangle}}
\newcommand{\bb}[1]{\ensuremath{\langle #1 \rangle}}

\newcommand{\MString}{\MO\bb8}

\newcommand{\MSO}{\mathtext{MSO}}
\newcommand{\MSpin}{\mathtext{MSpin}}
\newcommand{\BSpin}{\mathtext{BSpin}}
\newcommand{\BString}{\BO\bb8}
\newcommand{\tmf}{\mathtext{tmf}}
\newcommand{\susp}{\ensuremath{\Sigma}}

\newcommand{\ko}{\mathtext{ko}}

\newcommand{\sgn}{\mathtext{sgn}}

\newcommand{\pt}{\ensuremath{\mathtext{pt}}}

\newcommand{\til}[1]{\ensuremath{\widetilde{#1}}}

\newcommand{\mrm}[1]{\mathrm{#1}}
\newcommand{\mbf}[1]{\mathbf{#1}}

\newcommand{\Spin}{\ensuremath{\mathrm{Spin}}}
\newcommand{\Ff}{\ensuremath{\mathrm{F}_4}}
\newcommand{\B}{\ensuremath{\mathrm{B}}}
\newcommand{\BG}{\ensuremath{\mathrm{B}G}}
\newcommand{\BH}{\ensuremath{\mathrm{B}H}}
\newcommand{\BT}{\ensuremath{\mathrm{B}T}}
\newcommand{\ET}{\ensuremath{\mathrm{E}T}}
\newcommand{\Bi}{\ensuremath{\mathrm{Bi}}}

\newcommand{\s}{\ensuremath{\mathrm{s}}}
\newcommand{\p}{\ensuremath{\mathrm{p}}}
\newcommand{\W}{\ensuremath{\mathrm{W}}}

\newcommand{\inv}[2]{\big[\tfrac#1#2\big]\!}

\renewcommand{\H}{\mathtext{H}}

\newcommand{\iso}{\ensuremath{\cong}}

\newcommand{\tensor}{\otimes}

\newcommand{\Hom}{\mathtext{Hom}}

\newcommand{\into}{\ensuremath{\hookrightarrow}}

\newcommand{\ord}{\ensuremath{\mathrm{ord}}}

\theoremstyle{plain}
\newtheorem*{theorem*}{Theorem}
\newtheorem{theorem}{Theorem}
\newtheorem*{proposition*}{Proposition}
\newtheorem{proposition}[theorem]{Proposition}
\newtheorem{corollary}[theorem]{Corollary}
\newtheorem*{corollary*}{Corollary}
\newtheorem{lemma}[theorem]{Lemma}
\newtheorem*{lemma*}{Lemma}

\newtheorem*{exercise*}{Exercise}
\newtheorem{conjecture}[theorem]{Conjecture}
\newtheorem*{conjecture*}{Conjecture}

\newtheorem*{question*}{Question}

\theoremstyle{definition}

\newtheorem*{definition*}{Definition}

\newtheorem*{example*}{Example}
\newtheorem*{examples*}{Examples}
\newtheorem*{claim*}{Claim}
\theoremstyle{plain}

{

\newcommand{\xto}{\xrightarrow}

\allowdisplaybreaks[1]

\frenchspacing

\title{The Cayley plane and String bordism}

\author{Carl McTague}
\email{mctague@math.jhu.edu}
\urladdr{http://www.mctague.org/carl}
\address{Department of Mathematics, Johns Hopkins University, Baltimore, MD 21218, USA}
\address{Max-Planck-Institute f\"ur Mathematik, Vivatsgasse 7, 53111 Bonn, Germany}
\address{DPMMS, Cambridge University, Wilberforce Road, Cambridge CB3 0WB, England}

\begin{document}

\begin{abstract}
This paper shows that, away from 6, the kernel of the Witten genus is precisely the ideal consisting of (bordism classes of) Cayley plane bundles with connected structure group, but only after restricting the Witten genus to string bordism. It does so by showing that the divisibility properties of Cayley plane bundle characteristic numbers arising in Borel-Hirzebruch Lie-group-theoretic calculations correspond precisely to the divisibility properties arising in the Hovey-Ravenel-Wilson BP-Hopf-ring-theoretic calculation of string bordism at primes $>3$. 
\end{abstract}

\maketitle

\section*{Introduction}

This paper shows that an affinity between bordism rings and projective spaces extends further than previously known.

The first manifestation of the affinity is the fact that \emph{every positive-dimensional element of the unoriented bordism ring $\pi_*\MO$ is represented by a real projective bundle.}
In more detail, Thom \cite{thom54} showed that $\pi_*\MO$ is a polynomial ring over $\ZZ/2$ with one generator in each dimension not of the form $2^k-1$. Milnor \cite{milnor-1965} showed that a smooth degree--$(1,1)$ hypersurface $H \into \RP^i \times \RP^j$ can serve as generator if $1<i<j$ and if $\binom{i+j}i$ is not divisible by~$2$ (equivalently, if there are no `carries' when adding $i$ to $j$ in base $2$, see \cite[Probl.~16-F]{milnor-stasheff74}). If $i\le j$ then the projection $H \to \RP^i$ is a fiber bundle with fiber $\RP^{j-1}$. In fact, Stong \cite[Prop.~8.1]{stong-1973} showed that every positive-dimensional element of $\pi_*\MO$ is represented by an $\RP^2$ bundle. %

The second manifestation of the affinity is the fact that \emph{every  positive-dimensional element of the oriented bordism ring $\pi_*\MSO$ is represented by a complex projective bundle.}
In more detail, $\pi_*\MSO/\mrm{Torsion}$ is a polynomial ring over $\ZZ$ with one generator in each dimension $4k$. In each such dimension, a $\ZZ$-linear combination of smooth degree-$(1,1)$ hypersurfaces $H \into \CP^i \times \CP^j$ can serve as generator. If $i\le j$ then the projection $H \to \CP^i$ is a fiber bundle with fiber $\CP^{j-1}$. Wall \cite{wall-1960} showed that these generators, together with certain of Dold's \cite{dold-1956} generators for $\pi_*\MO$ (all of which are complex projective bundles), generate $\pi_*\MSO$. In fact, F\"uhring \cite{fuehring-2008} showed that every positive-dimensional element of $\pi_*\MSO$ is represented by a $\CP^2$ bundle. We shall return to this manifestation in more detail in the next section.

The third manifestation of the affinity is the fact that \emph{almost every element of the spin bordism ring $\pi_*\MSpin$ is represented by a quaternionic projective bundle}---specifically, the set of quaternionic projective bundles with connected structure group is an ideal of $\pi_*\MSpin$ (indeed, for any space $M$, if $F \to E \to B$ is a fiber bundle then so is $F \to E\times M \to B\times M$), and this ideal is precisely the kernel of the Atiyah invariant:
\begin{align*}
  \alpha : \pi_*\MSpin \to \pi_*\ko \iso
  \ZZ[\eta,\omega,\mu]/(2\eta,\eta^3,\eta\omega,\omega^2-4\mu)
\end{align*}
Here $\eta,\omega,\mu$ have degree 1,4,8 respectively. In more detail, Anderson-Brown-Peterson \cite{anderson-brown-peterson-67} computed $\pi_*\MSpin_{(2)}$ and the forgetful homomorphism $\pi_*\MSpin \to \pi_*\MSO$ becomes an isomorphism after inverting~2. Stolz \cite{stolz-1994} together with Kreck \cite{kreck-stolz-93} used this to show that every element of the kernel of the Atiyah invariant is represented by an $\HP^2$ bundle. ($\HP^2$ is 8-dimensional so $\HP^2$ bundles cannot possibly represent every element of the spin bordism ring.) The Atiyah invariant is thus a complete obstruction for the representability of a spin bordism class by an $\HP^2$ bundle. %

The fourth manifestation of the affinity is the subject of this paper: \emph{almost every element of the string bordism ring $\pi_*\MString$ is represented, at least up to powers of 2 and 3, by a Cayley plane---i.e.\ octonionic projective plane \emph{(}$\OP^2$\emph{)}---bundle.} (The Cayley plane is 16-dimensional so Cayley plane bundles cannot possibly represent every element of the string bordism ring.) Specifically, we prove that:

\begin{theorem}
  \label{thm:ker-witten-genus}
  Away from~6, the ideal of $\pi_*\MString$ consisting of (bordism classes of) Cayley plane bundles with connected structure group is precisely the kernel of the Witten genus. In other words, the extension of this ideal in $\pi_*\MString\inv16$ is precisely the kernel of:
  \begin{align*}
    \phi_{\mathrm{W}} \tensor \ZZ\inv16 \;:\; \pi_*\MString\inv16 \;\to\; \pi_*\tmf\inv16 \iso \ZZ\inv16[\GG_4,\GG_6]
  \end{align*}
  where $\GG_4, \GG_6$ have degree $8,12$ respectively.
\end{theorem}

The Witten genus is thus a complete obstruction for the representability of a string bordism class by a $\OP^2$ bundle, at least up to a powers of 2 and 3.

An interesting complication here is that Theorem~\ref{thm:ker-witten-genus} only appears to be true after restricting the Witten genus to string bordism. In other words, not every element of the kernel of the quasi-modular-form-valued Witten genus $\pi_*\MSO\inv16 \to \ZZ\inv16[\mathbf{G}_2, \mathbf{G}_4, \mathbf{G}_6]$ appears to be represented by a $\OP^2$ bundle. Far from it, in fact: the subring of $\pi_*\MSO\inv16$ generated by total spaces of oriented $\OP^2$ bundles (and string manifolds of dimension $<16$) appears to coincide with the image of the forgetful homomorphism $\pi_*\MString\inv16 \to \pi_*\MSO\inv16$. As we shall see, this homomorphism is the inclusion of an intricate, non-polynomial subring.

\bigskip

That Cayley plane bundles lie in the kernel of the Witten genus is already known:

\begin{theorem*}
  If $\OP^2 \to E \to W$ is a Cayley plane bundle with connected structure group then the Witten genus of $E$ vanishes.
\end{theorem*}

\label{90sHistory} This result was often proved in the 1990's---by Jung, Kreck-Singhof-Stolz, Dessai, H\"ohn---but rarely published. Rainer Jung's proof, which has yet to appear in print, used the work of Borel-Hirzebruch summarized below to show that the vanishing of the Witten genus on Cayley plane bundles is equivalent to the Jacobi triple identity for the Weierstrass sigma function. A little later Anand Dessai proved, using results of Kefeng Liu \cite{liu-1992}, that if $\mathrm{S}^3$ acts nontrivially on a string manifold $E$ then the Witten genus of $E$ vanishes. (This generalizes the theorem above since $\mathrm{S}^3$ acts nontrivially on the total space of any Cayley plane bundle.) Dessai's work appeared in the preprint \cite{dessai-1994}, in his PhD thesis \cite{dessai-1996}, and in the conference proceedings \cite{dessai-2009}. Around the same time Gerald H\"ohn proved, again using results of Liu, that the Witten genus of any string homogeneous manifold vanishes. These results helped inspire Stephan Stolz's conjecture \cite[see Thm.~3.1]{stolz96} that the Witten genus of a closed $4k$-dimensional string manifold vanishes if and only if it admits a Riemannian metric of positive Ricci curvature. (The author thanks Dessai for informing him of the history of these results.)

In fact, Jung and Dessai both proved the rational version of Theorem~\ref{thm:ker-witten-genus}:
\begin{theorem*}
  Rationally, the ideal of $\pi_*\MString$ consisting of (bordism classes of) Cayley plane bundles with connected structure group is precisely the kernel of the Witten genus. In other words, the extension of this ideal in $\pi_*\MString\tensor\QQ$ is precisely the kernel of:
  \begin{align*}
    \phi_{\mathrm{W}} \tensor \QQ \;:\; \pi_*\MString \tensor \QQ \;\to\; \pi_*\tmf \tensor \QQ \iso \QQ[\GG_4,\GG_6]
  \end{align*}
\end{theorem*}

Since stable rational homotopy theory is trivial, rational results are unsatisfying to homotopy theorists. This paper does not tackle the primes 2 or 3, the primes at which $\tmf$ is most interesting. But the author has no reason to be pessimistic about those primes and hopes that homotopy theorists will be pleased to see geometry in alignment at the primes $>3$. As far as the author knows, this paper gives the first geometrically explicit list of generators for $\pi_*\MString\inv16$.

\smallskip

Note that $\tmf\inv16$ is not a ring spectrum quotient of $\MString\inv16$. In~fact, for any prime $p>3$ and any sequence $X$ in $\pi_*\MString$, the $\pi_*\MString$-module: $$\pi_*\big(\MString_{(p)}/X\big)$$ is not (even abstractly) isomorphic to $\pi_*\tmf_{(p)}$ \cite{mctague-2013}.

\medskip

Throughout this paper the italic letter $p$ will denote a prime number. The roman letter~$\p$ will denote the Pontrjagin class.

\section{Pontrjagin numbers and oriented bordism}
\label{section:pontrjagin}

This section briefly reviews background material on Pontrjagin classes and the oriented bordism ring. This serves both to fix notation as well as to illustrate how the results of this paper extend well-known calculations.

The $i$th Pontrjagin class of a real vector bundle $V$ is by definition $\p_i(V)=(-1)^i\mrm{c}_{2i}(V \tensor \CC)$. It pulls back from the universal $i$th Pontrjagin class $\p_i$ in $\H^*(\BO(4n),\ZZ)$ for $n\ge i$, which in turn may be identified with the $i$th elementary symmetric polynomial. This is because the $i$th Pontrjagin class of a sum of complex line bundles is the $i$th elementary symmetric polynomial in the first Pontrjagin classes of the individual line bundles, $\p(L_1\oplus \cdots \oplus L_n)=\prod (1+\p_1(L_i))$. (The driving force behind this is the fact that, in ordinary cohomology, the total Chern class is exponential, $\mrm{c}(V_1 \oplus V_2)=\mrm{c}(V_1) \cdot \mrm{c}(V_2)$.)

It is a basic fact that the ring of symmetric polynomials is a polynomial ring on the elementary symmetric polynomials. There are other symmetric polynomials of geometric interest, though. Given a partition $I=i_1,\dots,i_r$ let $\s_I$ denote the polynomial $\sum \p_1(L_1)^{i_1} \cdots \p_1(L_r)^{i_r}$ where the sum runs over all distinct monomials obtained by permuting $L_1,\dots,L_n$. Each $\s_I$ is a symmetric polynomial, so may be written as a polynomial in the elementary symmetric polynomials. Thus we may associate to each $\s_I$ a polynomial in the Pontrjagin classes, which we also denote $\s_I$. Note in particular that $\s_1,\s_{1,1},\s_{1,1,1},\dots$ are the Pontrjagin classes $\p_1,\p_2,\p_3,\dots$ themselves. The geometric significance of the classes $\s_I$ comes from the following lemma (Lemma~16.2 of \cite{milnor-stasheff74}).

\pagebreak[3] %

\begin{lemma}[Thom]
  \label{lemma:thom}
  If $0 \to V_1 \to W \to V_2 \to 0$ is an exact sequence of vector bundles then:
  \begin{align*}
    \s_I(W) = \sum_{JK=I} \s_J(V_1) \; \s_K(V_2)
  \end{align*}
  where the sum ranges over all partitions  $J$ and $K$ with juxtaposition $JK$ equal to $I$.
\end{lemma}

This implies that $\s_n$ of the tangent bundle of a nontrivial product of closed oriented manifolds vanishes. In fact, a closed oriented manifold $M^{4n}$ is decomposable in $\pi_*\MSO\inv12$ if and only if the number $\s_n[M^{4n}]=\int_M\s_n(\mrm{T}M)$ equals zero. (The integral $\int_M$  here denotes the pushforward to a point $\H^{4n}(M) \to \H^0(\pt) \iso \ZZ$, equivalently, the Kronecker pairing $\langle \s_n(\mrm{T}M),[M]\rangle$ with the fundamental class $[M] \in \H_{4n}(M,\ZZ)$.) Since $\pi_*\MSO \tensor \QQ$ is a polynomial ring over $\QQ$ with one generator in each dimension $4n\ge4$, a sequence $\{M^{4n}\}_{n\ge1}$ therefore generates $\pi_*\MSO \tensor \QQ$ if and only if $\s_n[M^{4n}]\ne0$ for each $n\ge1$. As mentioned in the introduction, however, inverting just the prime 2 is enough to make $\pi_*\MSO$ a polynomial ring. It follows that the numbers $\s_n$ suffice to recognize a sequence of generators for $\pi_*\MSO\inv12$ but it turns out that these numbers have unexpected divisibility properties.

For any integer $n$ and any prime $p$ let $\ord_p(n)$ denote the $p$-adic order of $n$, that is, the largest integer $\nu$ such that $p^\nu$ divides $n$.

\begin{theorem}[{cf.~\cite[p.~180]{stong-1968}}]
  \label{thm:mso-hurewicz}
  A sequence $\{M^{4n}\}_{n\ge1}$ generates $\pi_*\MSO\inv12$ if and only~if:
 \begin{itemize}
 \item For any integer $n>0$ and any odd prime $p$:
   \begin{align*}
     \ord_p \big( \s_n[M^{4n}] \big) =
     \begin{cases}
       1 & \text{if $2n=p^i-1$ for some integer $i>0$} \\
       0 & \text{otherwise}
     \end{cases}
   \end{align*}
 \end{itemize}
\end{theorem}

Equivalently, if $p$ is odd then the Hurewicz homomorphism $\pi_*\MSO_{(p)} \to \H_*\MSO_{(p)}$, after passing to indecomposable quotients, is multiplication by $\pm p$ in degrees of the form $2(p^i-1)$ and is an isomorphism otherwise. (See \cite[Thm.~3.1.5]{ravenel-1986} where the special behavior in degrees $2(p^i-1)$ ultimately comes from the degrees of the generators $v_i$ of $\pi_*\BP$.)

\medskip

Now we return to the second manifestation of the affinity discussed in the introduction.

\begin{proposition*} %
  If $H \into \CP^i \times \CP^{2n-i+1}$ is a smooth complex hypersurface of degree $(1,1)$ and $1<i<2n$ then:
  \begin{align*}
    \s_n[H] = - \binom{2n+1}{i}
  \end{align*}
\end{proposition*}
\begin{proof}
  Since the tangent bundle of the ambient manifold $\CP^i \times \CP^{2n-i+1}$ splits nontrivially, Lemma~\ref{lemma:thom} implies that $\s_n(\mrm{T}H)=-\s_n (\mrm{N}H)$ where the normal bundle $\mrm{N}H$ is isomorphic to the complex line bundle:
  \begin{align*}
    \mrm{O}(1,1)|_H=\big(\pi_1^*\mrm{O}(1) \tensor \pi_2^*\mrm{O}(1)\big)\big|_H
  \end{align*}
  and $\pi_1,\pi_2$ are the projections of the ambient manifold. Since for a complex line bundle $\p_1=\mrm{c}_1^2$ and since in ordinary cohomology $\mrm{c}_1(L_1 \tensor L_2)=\mrm{c}_1(L_1)+\mrm{c}_1(L_2)$, it follows that:
  \begin{align*}
    \s_n(\mrm{O}(1,1))
    =\p_1(\mrm{O}(1,1))^n
    =\mrm{c}_1(\mrm{O}(1,1))^{2n}
    =(x_1+x_2)^{2n}
  \end{align*}
  where $x_j=\pi_j^*\mrm{c}_1(\mrm{O}(1))$. Thus:
  \begin{align*}
    \s_n[H] &= -\int_H \s_n(\mrm{O}(1,1)|_H)
    = -\int_H \left(x_1+x_2\right)^{2n}\Big|_H
    \intertext{By Poincar\'e duality then (see \cite[Prob.~16-D]{milnor-stasheff74}):}
    \s_n[H] &= -\int_{\CP^i \times \CP^{2n-i+1}} (x_1+x_2)^{2n+1} = -\binom{2n+1}i \qedhere
  \end{align*}
\end{proof}

Kummer's theorem, which states that $\ord_p\big[\binom{n}{i}\big]$ equals the number of `carries' when adding $i$ to $n-i$ in base $p$ 
(see \cite[\S1]{granville97}), can be used to show that:

\begin{lemma*}
  For any integer $n>0$ and any odd prime $p$:
  \begin{align*}
    \ord_p \left[ \mathop{\mathrm{GCD}}_{1<i<2n} \binom{2n+1}{i} \right] =
    \begin{cases}
      1 & \text{if $2n+1=p^i$ for some integer $i>0$} \\
      0 & \text{otherwise}
    \end{cases}
  \end{align*}
\end{lemma*}

It follows that $\ZZ$-linear combinations of the hypersurfaces appearing in the proposition generate $\pi_*\MSO\inv12$, as asserted in the introduction.

\smallskip

In short, then, the divisibility properties of $\s_n$ for oriented manifolds, deduced from homotopy theory, align perfectly with the divisibility properties of $\s_n$ for $\CP^n$ bundles, deduced from divisibility properties of binomial coefficients.

This paper will follow the same outline. First we will deduce the divisibility properties of $\s_n$ (and $\s_{n,n'}$) for string manifolds from known results in homotopy theory. Then we will show that these divisibility properties align perfectly with the divisibility properties of $\s_n$ (and $\s_{n,n'}$) for Cayley plane bundles, which we will in turn deduce from divisibility properties of binomial coefficients. The arguments and calculations will at each stage be more complicated than for oriented bordism and complex projective bundles, but the outline and spirit will be the same.

\section{How to recognize generators for string bordism}
\label{section:homotopy}

In the preceding section we stated a criterion (Theorem~\ref{thm:mso-hurewicz}), involving the number $\s_n$, which ensures that a sequence $\{M^{4n}\}_{n\ge1}$ generates $\pi_*\MSO\inv12$. The purpose of this section is to establish an analogous criterion (Theorem~\ref{thm:mstring-gen}) for the string bordism ring $\pi_*\MString\inv16$. It turns out that Pontrjagin numbers still suffice to distinguish elements of $\pi_*\MString\inv16$ but, since this ring is not a polynomial ring, the numbers $\s_n$ do not suffice to recognize generators; certain numbers of the form $\s_{n,n'}$ are also needed. As we shall see, the criterion is a consequence of Hovey's calculation \cite{hovey-2008} of $\pi_*\MString_{(p)}$ for $p>3$.

\smallskip

First recall what string bordism is. Any real vector bundle $V \to X$ of rank $k$ pulls back from the universal rank-$k$ bundle over the classifying space $\BO(k)$ by a map $f:X \to \BO(k)$. 
\begin{wrapfigure}{r}{100pt}
  $
  \xymatrix@C40pt@R12pt{
    & \BO\bb8 \ar[d] \\
    & \BSpin \ar[d] \\
    & \BSO \ar[d] \\
    X \ar[r]^f \ar@(ur,l)@{-->}[ur]^{f_2} \ar@(u,l)@{-->}[uur]^{f_4} \ar@(ul,l)@{-->}[uuur]^{f_8} & \BO
  }
  $
\end{wrapfigure}

\vspace{-1.2em}
\begin{itemize}
\item An \emph{orientation} of $V$ is a (homotopy class of) lift $f_2$ of $f$ to the 1-connected cover $\BSO\to\BO$. Such a lift exists if and only if the generator $\mrm{w}_1$ of $\H^1(\BO,\ZZ/2)$ pulls back to 0 in $\H^1(X,\ZZ/2)$.
\item A \emph{spin structure} on $V$ is a (homotopy class of) lift $f_4$ of $f_2$ to the 3-connected cover $\BSpin\to\BSO$. Such a lift exists if and only if the generator $\mrm{w}_2$ of $\H^2(\BSO,\ZZ/2)$ pulls back to 0 in $\H^2(X,\ZZ/2)$.
\item A \emph{string structure} on $V$ is a (homotopy class of) lift $f_8$ of $f_4$ to the 7-connected cover $\BString\to\BSpin$. Such a lift exists if and only if the generator $\tfrac12\p_1$ of $\H^4(\BSpin,\ZZ)$ pulls back to 0 in $\H^4(X,\ZZ)$.
\end{itemize}

The bordism spectrum of string manifolds $\MString$ is the Thom spectrum of the map $\BString \to \BO$. Its coefficient ring $\pi_*\MString$ is the bordism ring of manifolds equipped with a string structure on their stable normal bundle.

\pagebreak[3] %

\begin{theorem}
  \label{thm:mstring-gen}
  A set $S$ generates $\pi_*\MString\inv16$ if:
  \begin{enumerate}
  \item For each integer $n>1$, there is an element $M^{4n}$ of $S$ such that for any prime $p>3$:
    \begin{align*}
      \ord_p \big( \s_n[M^{4n}] \big) =
      \begin{cases}
        1 & \text{if $2n=p^i-1$ or $2n=p^i+p^j$ for some integers $0 \le i \le j$} \\
        0 & \text{otherwise}
      \end{cases}
    \end{align*}
  \item For each prime $p>3$ and each pair of integers $0<i<j$, there is an element $N^{2(p^i+p^j)}$ of $S$ such that:
    \begin{align*}
      \s_{(p^i+p^j)/2}[N^{2(p^i+p^j)}]&=0
      \intertext{but:}
      \s_{(p^i+1)/2,(p^j-1)/2}[N^{2(p^i+p^j)}] &\not\equiv 0 \mod p^2
    \end{align*}
  \end{enumerate}
\end{theorem}

\medskip

We prove this in stages.

\begin{proposition*}
  The forgetful homomorphism:
  \begin{align*}
    \pi_*\MString\inv16 \to \pi_*\MSpin\inv16
  \end{align*}
  is injective.
\end{proposition*}
\begin{proof}
  It is injective tensor $\QQ$ so its kernel is torsion (since $\QQ$ is a flat $\ZZ$-module). Giambalvo, however, showed that $\pi_*\MString$ has no $p$--torsion for $p>3$ \cite[Thm.~4.3]{giambalvo-71}.
\end{proof}

Since $\pi_*\MSpin\inv12 \iso \pi_*\MSO\inv12$:

\begin{corollary*}
  Any two string structures for an oriented manifold determine the same element of $\pi_*\MString\inv16$.
\end{corollary*}

And since Pontrjagin numbers detect equality in $\pi_*\MSO\inv12$:

\begin{corollary*}
  Pontrjagin numbers detect equality in $\pi_*\MString\inv16$.
\end{corollary*}

To prove Theorem~\ref{thm:mstring-gen} it therefore suffices to determine the image of $\pi_*\MString\inv16 \to \pi_*\MSpin\inv16$ or, equivalently, to determine the image of $\pi_*\MString_{(p)} \to \pi_*\MSpin_{(p)}$ for each prime $p>3$. The Hovey-Ravenel-Wilson approach \cite{ravenel-wilson-1974,hovey-ravenel-1995} to $\BO\bb{4k}$ reduces $\pi_*\MString_{(p)} \to \pi_*\MSpin_{(p)}$ to the homomorphism $\BP_*\BBPb1_{2(p+1)} \to \BP_*\BBPb1_4$, and Hovey's description \cite{hovey-2008} of these rings reveals enough information about the image to prove Theorem~\ref{thm:mstring-gen}. What follows is a brief summary of the results of \cite{ravenel-wilson-1974,hovey-ravenel-1995,hovey-2008} needed to prove Theorem~\ref{thm:mstring-gen}.

\bigskip

First some standard notation. Let $\BP$ denote the Brown-Peterson spectrum \cite{brown-peterson-1966}; its coefficient ring is $\pi_*\BP\iso\ZZ_{(p)}[v_1,v_2,\dots]$ where $\deg(v_i)=2(p^i-1)$. Let $\BP\bb1$ denote the Johnson-Wilson spectrum obtained from $\BP$ by killing the ideal $(v_2,v_3,\dots)$ of $\pi_*\BP$; its coefficient ring is $\pi_*\BP\bb1\iso\ZZ_{(p)}[v_1]$ and its homotopy type is independent of the polynomial generators $v_2,v_3,\dots$ chosen \cite{johnson-wilson-1973}. The infinite loop space obtained by applying the $k$-th space functor to a spectrum \textit{X} will be denoted \emph{\textbf{X}}$_k$. 

\smallskip

Recall that the ring homomorphism $\pi_*\BP\bb1 \to \pi_*\ku_{(p)}$ taking $v_1$ to $v^{p-1}$ lets one identify $\pi_*\ku_{(p)} \iso \ZZ_{(p)}[v]$ with $\pi_*\BP\bb1[v]/(v_1-v^{p-1})$. This identification extends to a multiplicative splitting of spectra:
\begin{align*}
  \ku_{(p)} \iso \prod_{i=1}^{p-2} \susp^{2i} \BP\bb1
\end{align*}
Multiplication by $v$ on the left corresponds to the (upward) shift of factors on the right, the shift from top to bottom factor being accompanied by multiplication by $v_1$.

Since, for $k$ even, $\BU\bb{k}$ can be taken as the $k$-th space of $\ku$, this implies that there is a $p$-local decomposition of H-spaces:
\begin{align*}
  \BU\bb k_{(p)} \iso \prod_{i=1}^{p-2} \BBPb1_{k+2i}
\end{align*}

There is an analogous splitting of $\BO\bb k_{(p)}$ for $p>2$:

\begin{theorem*}[{\cite[Cor.~1.5]{hovey-ravenel-1995}}]
  If $k$ is divisible by $4$ and $p>2$ then there is a $p$-local decomposition of H-spaces:
  \begin{align*}
    \BO\bb k_{(p)} \iso \prod_{i=0}^{(p-3)/2} \BBPb1_{k+4i}
  \end{align*}
  Under this decomposition the map $\BO\bb{k+4}\to\BO\bb{k}$ corresponds to the identity map on the factors $\BBPb1_{k+4i}$ for $0<i<\tfrac12(p-3)$ and to $[v_1]:\BBPb1_{k+2p-2}\to\BBPb1_k$ on the remaining factor.
\end{theorem*}
If $k=4$ then the situation looks like this:
\begin{align*}
  \xymatrix@C5pt@R25pt{ \BO\bb8_{(p)} \ar[d] &\iso &\BBPb1_8 \ar[drr]
    &\times &\BBPb1_{12} \ar[drr] &\times & \cdots & & \cdots
    \ar[drr] &\times &\BBPb1_{2p+2} \ar@(d,u)@{-->}[dllllllll] \\
    \BO\bb4_{(p)} &\iso &\BBPb1_4 &\times & \BBPb1_{8} &\times &
    \BBPb1_{12} & \times & \cdots & \times & \BBPb1_{2p-2} }
\end{align*}

\smallskip

Hovey shows that $\pi_*\MString_{(p)}$ is (abstractly) isomorphic as a ring to a quotient of the $\BP$-homology of this decomposition, the ring structure of the latter coming from the infinite loop space structures of the factors. To state his result precisely, we need to introduce some notation. If $p>2$ then there is a natural map of ring spectra $\MString \to \MSO \to \BP$. If $p>3$ then the induced homomorphism $\BP_*\MString \to \BP_*\BP$ is surjective \cite[Lemma~2.1]{hovey-2008}. For each positive integer $i$, choose a generator $u_i$ in $\BP_{2(p^i-1)}\MString$ mapping to the generator $t_i$ of $\BP_*\BP \iso \BP_*[t_1,t_2,\dots]$. For dimension reasons each $u_i$ must lie in the tensor factor $\BP_*\BBPb1_{2p-2}$ of $\BP_*\MString$.

\begin{theorem*}[{\cite[Thm.~2.4]{hovey-2008}}]
  If $p>3$ then there are (abstract) isomorphisms of rings:
  \begin{align*}
    \pi_*\MString_{(p)} \iso %
    \BP_*\BBPb1_8 &\tensor_{\BP_*} \BP_*\BBPb1_{12} \\
    \tensor_{\BP_*}\cdots &\tensor_{\BP_*} \BP_*\BBPb1_{2p-2}
    / (u_1,u_2,\dots) \tensor_{\BP_*} \BP_*\BBPb1_{2p+2} \\
    \pi_*\MSpin_{(p)} \iso \BP_*\BBPb1_4 &\tensor_{\BP_*} \cdots
    \tensor_{\BP_*} \BP_*\BBPb1_{2p-2}/(u_1,u_2,\dots)
  \end{align*}
\end{theorem*}

So to understand the forgetful homomorphism $\pi_*\MString_{(p)} \to \pi_*\MSpin_{(p)}$ it suffices to understand the ring homomorphism induced by the dotted arrow above:
\begin{align*}
  [v_1]_* : \BP_*\BBPb1_{2p+2} \to \BP_*\BBPb1_4
\end{align*}
As we shall see, it is the inclusion of a non-polynomial subring into a polynomial ring. A toy model worth bearing in mind is the inclusion $\ZZ[5x,y,xy] \into \ZZ[x,y]$.

\medskip

Instead of studying each ring $\BP_*\BBPb1_n$ individually, Hovey exploits the fact that they fit together to form a Hopf ring $\BP_*\BBPb1_*$.  In particular there is a circle product:
\begin{align*}
  \circ : \BP_*\BBPb1_m \tensor \BP_*\BBPb1_n \to \BP_*\BBPb1_{m+n}
\end{align*}
corresponding to the ring spectrum structure of $\BP\bb1$. It gives an inductive way to construct elements in the increasingly complicated rings $\BP_*\BBPb1_{m+n}$. In fact, all the elements we will need can be constructed that way from just two kinds of elements, $b_{(i)}$ and $[v_1^i]$, defined as follows. The complex orientation gives a map $\CP^\infty \to \BBPb1_2$. Let $b_i \in \BP_{2i}\BBPb1_2$ be the image under this map of the $\BP$-homology generator of degree $2i$. Let $b_{(i)}$ denote the generator $b_{p^i}$ (generators not of this form are decomposable). The homotopy class $v_1^i$ is represented by a map $\mathrm{S}^0 \to \BBPb1_{-2i(p-1)}$. Let $[v_1^i] \in \BP_0\BBPb1_{-2i(p-1)}$ denote the image under this map of the $\BP$-homology generator. 

Wilson \cite[Cor.~5.1]{wilson-1975} showed that, for $n<2p+2$, the $p$-local homology of $\BBPb1_n$ is an evenly graded torsion-free polynomial algebra with one generator in each dimension corresponding to $s^nv_1^k$ for $k\ge0$. The Atiyah-Hirzebruch spectral sequence therefore collapses and the $\BP$ homology of $\BBPb1_n$ has the same properties. In fact:

\begin{theorem*}[{\cite[Thm.~1.2]{hovey-2008}}]
  If $n<2p$ then $\BP_*\BBPb1_n$ is a polynomial algebra over $\BP_*$ with one generator in each positive even degree congruent to $n$ mod~$2p-2$. In a degree $2m$ of that form, one can take as generator:
  \begin{align*}
    x_{2m} = [v_1^i] \circ b_{(0)}^{\circ j_0} \circ b_{(1)}^{\circ
      j_1} \circ \cdots \circ b_{(k)}^{\circ j_k}
  \end{align*}
  where $m=\sum j_lp^l$ is the $p$-adic expansion and $i=\tfrac1{p-1}(\alpha(m)-\tfrac12n)$ with $\alpha(m)=\sum_l j_l$.
\end{theorem*}

\smallskip

If $n=2p+2$ then $\BP_*\BBPb1_n$ is \emph{not} a polynomial ring over $\BP_*$. It has a generator in each degree congruent to $4$ mod~$2p-2$ (and greater than $4$) but it has two generators in some of these dimensions, and these generators satisfy a relation. Specifically:
\begin{itemize}
\item In each degree $4p^i$ for $i>0$ there is one generator:
  \begin{align*}
    w_{4p^i} &= b_{(i)} \circ b_{(i-1)}^{\circ p}
    \intertext{\item In each degree $2(p^i+p^j)$ for $0\le i<j$ there is a generator:}
    y_{2(p^i+p^j)} &= b_{(i)} \circ b_{(j-1)}^{\circ p} \\
    \intertext{\item In each degree $2(p^i+p^j)$ for $0<i<j$ there is a second generator:}
    z_{2(p^i+p^j)} &= b_{(i-1)}^{\circ p} \circ b_{(j)}
  \end{align*}
  To simplify formulas later on, let $z_{2(1+p^j)}=0$ for $j>0$.
\item In each of the other degrees---that is, in each degree $2m$ congruent to $4$ mod~$2p-2$ but not of the form $2(p^i+p^j)$ for any $0\le i \le j$---there is a single generator of the form $x_{2m}$, defined as in the preceding theorem.
\end{itemize}

Hovey constructs, for each $0<i<j$, a relation $r_{ij}$ involving $y_{2(p^i+p^j)}$, $z_{2(p^i+p^j)}$ and $p$. To express it, let $\mrm{I}$ be the ideal of $\BP_*$ generated by $(p,v_1,v_2,\dots)$ and let $\mrm{I}(n)$ be the kernel of $\BP_*\BBPb1_n \to \BP_*$.

\begin{proposition*}[{\cite[Cor.~1.6]{hovey-2008}}]
  For any pair of integers $0<i<j$ there is a relation in $\BP_*\BBPb1_{2p+2}$ of the form:
  \begin{multline*}
    p(z_{2(p^i+p^j)}-y_{2(p^i+p^j)}) \equiv 
    v_j y_{2(1+p^i)} - v_i \cdot y_{2(1+p^j)}
    + y_{2(p^{i-1}+p^{j-1})}^p - z_{2(p^{i-1}+p^{j-1})}^p 
    \\ \mod \mrm{I}^2 \cdot \mrm{I}(2p+2) + \mrm{I} \cdot
    \mrm{I}(2p+2)^{*2} + \mrm{I}(2p+2)^{*p+1}
  \end{multline*}
\end{proposition*}

Considering each of these relations as an element $r_{ij}$ of the $\BP_*$-polynomial ring $\mrm{R}$ on all the generators $w_{4p^i},y_{2(1+p^i)},y_{2(p^i+p^j)},z_{2(p^i+p^j)},x_{2m}$ for
$0<i<j$ and $2m$ of the form described above, Hovey shows that:
\begin{theorem*}[{\cite[Thm.~1.7]{hovey-2008}}]
  \begin{align*}
  \mrm{R}/(r_{ij} \;|\; 0<i<j) \to \BP_*\BBPb1_{2p+2}
\end{align*}
is an isomorphism of $\BP_*$--algebras.
\end{theorem*}

\medskip

Remember that we want to understand the homomorphism:
\begin{align*}
  [v_1]_* : \BP_*\BBPb1_{2p+2} \to \BP_*\BBPb1_4
\end{align*}
If $0<i<j$ then by definition:
\begin{align*}
  \begin{cases}
    [v_1]_* \; w_{4p^i} &= [v_1] \circ
    b_{(i)} \circ b_{(i-1)}^{\circ p} \\
    [v_1]_* \; y_{2(1+p^i)} &= [v_1]
    \circ b_{(0)} \circ b_{(i-1)}^{\circ p} \\
    [v_1]_* \; y_{2(p^i+p^j)} &= [v_1]
    \circ b_{(i)} \circ b_{(j-1)}^{\circ p} \\
    [v_1]_* \; z_{2(p^i+p^j)} &= [v_1] \circ b_{(i-1)}^{\circ p} \circ
    b_{(j)} \\
    [v_1]_* \; x_{2m} &= \underbrace{[v_1] \circ
      [v_1^i]}_{\hphantom{[v_1^{i+1}]}=[v_1^{i+1}]} \circ b_{(0)}^{\circ j_0}
    \circ b_{(1)}^{\circ j_1} \circ \cdots \circ b_{(k)}^{\circ j_k}
  \end{cases}
\end{align*}
Recall that the exponent $i$ of $v_1$ appearing in the generator $x_{2m}$ depends on both $m$ and $n$, specifically $i=i(m,n)=\tfrac1{p-1}(\alpha(m)-\tfrac12n)$. So $i(m,4)=i(m,2p+2)+1$ and the homomorphism carries each generator of $\BP_*\BBPb1_{2p+2}$ of the form $x_{2m}$ to the corresponding generator $x_{2m}$ of
$\BP_*\BBPb1_4$. To relate the images of the other generators to the generators $x_{2m}$ of $\BP_*\BBPb1_4$, we rely on the following proposition.

\begin{proposition*}[{\cite[Cor.~1.5]{hovey-2008}}]
  For each integer $i>0$ there is a relation in $\BP_*\BBPb1_2$ of the form:
  \begin{align*}
    [v_1] \circ b_{(i-1)}^{\circ p} \equiv v_i \cdot b_{(0)} - p
    \cdot b_{(i)} - b_{(i-1)}^{*p} \mod \mrm{I}^2 \cdot \mrm{I}(2) +
    \mrm{I} \cdot \mrm{I}(2)^{*2} + \mrm{I}(2)^{*p+1}
  \end{align*}
\end{proposition*}
If we $\circ$-multiply this relation by $b_{(j)}$ then we obtain a relation in $\BP_*\BBPb1_4$:
\newlength{\mylength}
\settowidth{\mylength}{$_{\circ b_{(j-1)})^{*p}}$}
\begin{multline*}
  [v_1] \circ b_{(i-1)}^{\circ p} \circ b_{(j)} \equiv v_i \cdot
  b_{(0)} \circ b_{(j)} - p \cdot b_{(i)} \circ b_{(j)}
  -\hspace{-\mylength}\underbrace{b_{(i-1)}^{*p} \circ
    b_{(j)}}_{\hphantom{(b_{(i-1)} \circ b_{(j-1)})^{*p}}=(b_{(i-1)} \circ b_{(j-1)})^{*p}}
  \\ \mod \mrm{I}^2 \cdot \mrm{I}(4) + \mrm{I} \cdot I(4)^{*2} + \mrm{I}(4)^{*p+1}
\end{multline*}
The bracketed equality is a consequence of the Hopf ring distributive law (see the discussion just before Lemma~1.7 of \cite{hovey-ravenel-1995}). If $j=0$ then (as that discussion points out) the bracketed quantity equals $0$. The fact that $\BP_*\BBPb1_m \circ \mrm{I}(n)^{*k} \subseteq \mrm{I}(n+m)^{*k}$ is also a consequence of the Hopf ring distributive law.

Substituting $(i,j) \mapsto (i,i),(1,i),(j,i),(i,j)$ (and subtracting) produces, for $0<i<j$, the following congruences mod~$\mrm{I}^2 \cdot \mrm{I}(4) + \mrm{I} \cdot \mrm{I}(4)^{*2} + \mrm{I}(4)^{*p+1}$:
\begin{align*}
  \begin{cases}
    [v_1]_*\, w_{4p^i} & \equiv v_i \cdot x_{2(1+p^i)} - p \cdot x_{4p^i} - x_{4p^{i-1}}^p \\
    [v_1]_*\, y_{2(1+p^i)} & \equiv v_i \cdot x_4 - p \cdot x_{2(1+p^i)} \\
    [v_1]_*\, y_{2(p^i+p^j)} & \equiv v_j \cdot x_{2(1+p^i)} - p \cdot x_{2(p^i+p^j)} - x_{2(p^{i-1}+p^{j-1})}^p \\
    [v_1]_*(z_{2(p^i+p^j)}-y_{2(p^i+p^j)}) &\equiv v_i \cdot x_{2(1+p^j)} - v_j \cdot x_{2(1+p^i)}
  \end{cases}
\end{align*}

\medskip

These congruences suffice for computing characteristic numbers of the form $\s_n$ and $\s_{n,n'}$ since, by Lemma~\ref{lemma:thom}, such numbers vanish on the ideal $\mrm{I}^2 \cdot \mrm{I}(4) + \mrm{I} \cdot \mrm{I}(4)^{*2} + \mrm{I}(4)^{*p+1}$.

To compute $\s_n$ and $\s_{n,n'}$ of the right hand sides of these congruences note that, by the construction of $u_i$, the image of $v_i$ in:
\begin{align*}
  \pi_*\MSpin_{(p)} \iso \BP_*\BBPb1_4 \tensor_{\BP_*} \cdots
  \tensor_{\BP_*} \BP_*\BBPb1_{2p-2}/(u_1,u_2,\dots)
\end{align*}
can serve as the $\ZZ_{(p)}$--polynomial algebra generator of degree $2(p^i-1)$. So by Theorem~\ref{thm:mso-hurewicz}, $p$ divides $\s_{(p^i-1)/2}[v_i]$ to order~$1$. Similarly, if $2m$ is not of the form $2(p^i-1)$ then the image of $x_{2m}$ may serve as the $\ZZ_{(p)}$--polynomial algebra generator of degree $2m$. So by Theorem~\ref{thm:mso-hurewicz}, $p$ does not divide $\s_{m/2}[x_{2m}]$

Thus, by Lemma~\ref{lemma:thom}:
\begin{align*}
  \s_{p^i}\left( [v_1]_*w_{4p^i}\right) &= \s_{p^i}\left(  v_i \cdot x_{2(1+p^i)} - p \cdot x_{4p^i} - x_{4p^{i-1}}^p \right) \\
  &= \underbrace{\s_{p^i} \left( v_i \cdot x_{2(1+p^i)} \right)}_{=0} - p \cdot \s_{p^i} \left( x_{4p^i} \right) - \underbrace{\s_{p^i}\left(x_{4p^{i-1}}^p\right)}_{=0}
\end{align*}
and since:
\begin{align*}
  \ord_p(a\cdot b) &= \ord_p(a)+\ord_p(b)
\end{align*}
it follows that:
\begin{align*}
  \ord_p \big[ \s_{p^i} \left( [v_1]_*w_{4p^i}\right) \big] &= \underbrace{\ord_p\big[p\big]}_{=1}+\underbrace{\ord_p\left[\s_{p^i}\left(x_{4p^i}\right)\right]}_{=0} = 1
\end{align*}

Similarly, by Lemma~\ref{lemma:thom}:
\begin{align*}
  \s_{(p^i+1)/2,(p^j-1)/2}\big(z_{2(p^i+p^j)}-y_{2(p^i+p^j)}\big)
  &=\s_{(p^i+1)/2,(p^j-1)/2}\big( v_i \cdot x_{2(1+p^j)} - v_j \cdot x_{2(1+p^i)} \big) \\
  &=-\s_{(p^j-1)/2}\big(v_j\big) \cdot \s_{(p^i+1)/2}\big(x_{2(1+p^i)}\big)
\end{align*}
and it follows that:
\begin{align*}
  \ord_p &\left[ \s_{(p^i+1)/2,(p^j-1)/2}\big(z_{2(p^i+p^j)}-y_{2(p^i+p^j)}\big) \right] \\ &= \underbrace{\ord_p\left[\s_{(p^j-1)/2}\big(v_j\big)\right]}_{=1} + \underbrace{\ord_p\left[\s_{(p^i+1)/2}\big(x_{2(1+p^i)}\big)\right]}_{=0} = 1
\end{align*}

These and similar calculations show that $p$ divides:
\begin{align*}
  \s_{p^i}\big([v_1]_*w_{4p^i}\big) &\text{\;to order 1} \\
  \s_{(1+p^i)/2}\big([v_1]_*y_{2(1+p^i)}\big) &\text{\;to order 1} \\
  \s_{(p^i+p^j)/2}\big([v_1]_*y_{2(p^i+p^j)}\big) &\text{\;to order 1} \\
  \s_{(p^i+p^j)/2}\big([v_1]_*(z_{2(p^i+p^j)}-y_{2(p^i+p^j)})\big)
  &\text{\;to order $\infty$} \\
  \text{and\;}
  \s_{(p^i+1)/2,(p^j-1)/2}\big(z_{2(p^i+p^j)}-y_{2(p^i+p^j)}\big)
  &\text{\;to order 1} \\
  \text{but\;}
  \s_{(p^i+1)/2,(p^j-1)/2}\big([v_1]_* (v_j \cdot y_{2(1+p^i)})\big)
  &\text{\;to order 2}
\end{align*}
(Recall that by definition $\ord_p(0)=\infty$.)

Theorem~\ref{thm:mstring-gen} follows from these six facts, (1) from the first three and (2) from the last three. In more detail, the last three facts imply that the image of $z_{2(p^i+p^j)}-y_{2(p^i+p^j)}$ can be distinguished from the image of $y_{2(p^i+p^j)}$ and from the images of degree-$2(p^i+p^j)$ products of lower degree generators by the vanishing of the number $\s_{(p^i+p^j)/2}$ together with the nonvanishing mod~$p^2$ of the number $\s_{(p^i+1)/2,(p^j-1)/2}$.

\section{Cayley plane bundles}

In this section we summarize work of Borel \& Hirzebruch \cite{borel-hirzebruch-1958,borel-hirzebruch-1959} on characteristic classes of homogeneous spaces which we will use in the next section to prove Theorem~\ref{thm:ker-witten-genus}.

The Cayley plane is the homogeneous space $\OP^2=\Ff/\Spin(9)$. Much of what follows applies to any bundle with fiber a homogeneous space $G/H$, so we begin in that generality and later specialize to the case $G/H=\Ff/\Spin(9)$. 

Throughout this section let $G$ be a compact connected Lie group, let $\mrm{i}_{H,G} : H \into G$ be a maximal rank subgroup, and let $\mrm{i}_{T,H} : T \to H$ and $\mrm{i}_{T,G} : T \to G$ be a common maximal torus:
\begin{align*}
  \xymatrix@R15pt{ & H \ar[d]^{\mrm{i}_{H,G}} \\
    T \ar[ur]^{\mrm{i}_{T,H}} \ar[r]_{\mrm{i}_{T,G}} & G
  }
\end{align*}

Every $G/H$ bundle (with structure group $G$) pulls back from the universal $G/H$ bundle $\BH \to \BG$. That is, every $G/H$ bundle (with structure group $G$) fits into a pullback square:
\begin{align*}
  \xymatrix{ E \ar[r]^-{\til{g}} \ar[d]_\pi
    & \BH \ar[d]^{\Bi_{H,G}} \\
    Z \ar[r]^-g & \BG }
\end{align*}
where $g$ is unique up to homotopy. %

Let $\eta$ denote the bundle of tangents along the fibers of $\BH \to \BG$. Then the bundle of tangents along the fibers of $E \to Z$ is the pullback $\til{g}^*(\eta)$ and there is an exact sequence:
\begin{align*}
  0 \to \til{g}^*(\eta) \to \mrm{T}E \to \pi^* \mrm{T}Z \to 0
\end{align*}
This enables us to compute the characteristic classes of $\mrm{T}E$ from those of $\eta$ and $\mrm{T}Z$, e.g.:
\begin{align*}
  \p_1(\mrm{T}E)=\til{g}^*\p_1(\eta) + \pi^*\p_1(\mrm{T}Z)
\end{align*}

The characteristic classes of $\eta$, or rather their pullbacks to $\H^*(\BT,\ZZ)$, may in turn be computed using the beautiful methods of Borel-Hirzebruch.
To state their results precisely, we need to introduce some notation (see \cite[Ch.~1]{borel-hirzebruch-1958} for more detail).

\medskip

Let $V$ be the universal cover of the maximal torus $T$. Let $\Gamma$ be the unit lattice of $V$, i.e.\ the inverse image of the identity element of $T$. A real valued linear form on $V$ is called \emph{integral} if it takes integral values on $\Gamma$; the group of all such forms $\Hom(\Gamma,\ZZ)$ is naturally isomorphic to $\H^1(T,\ZZ)\iso\Hom(\pi_1(T),\ZZ)$. The adjoint representation of $T$ on the Lie algebra $\mathfrak{g}$ of $G$ is fully reducible, and there is a direct sum decomposition of $\mathfrak{g}$ into invariant subspaces:
\begin{align*}
  \mathfrak{g} = \mathfrak{a}_1 + \cdots + \mathfrak{a}_m + \mathfrak{t}
\end{align*}
where $\dim(\mathfrak{a}_i)=2$. The action on $\mathfrak{a}_i$ of an element $t$ of $T$ may be written:
\begin{align*}
  \begin{pmatrix}
    \cos 2\pi a_i(t) & -\sin 2\pi a_i(t) \\
    \sin 2\pi a_i(t) & \quad \cos 2\pi a_i(t)
  \end{pmatrix}
\end{align*}
The function $a_i:T\to\RR$ lifts to a nonzero integral linear form on $V$, also denoted $a_i$. The linear forms $\pm a_1,\dots,\pm a_m$ on $V$ are called the \emph{roots} of $G$. The decompositions of $\mathfrak{g}$ and $\mathfrak{h}$ may be chosen compatibly so that we may speak of the roots $\pm \bar{a}_1,\dots,\pm \bar{a}_k$ of $G$ \emph{complementary} to those of $H$.

Transgression in a principal $T$-bundle $P\to P/T$ associates to each element of $\H^1(T,\ZZ)$ an element of $\H^2(P/T,\ZZ)$. Since $\H^1(T,\ZZ)\iso\Hom(\Gamma,\ZZ)$ (as discussed above), this associates to each root of $G$, and more generally to each integral form, an element of $\H^2(P/T,\ZZ)$. For the universal $T$-bundle $\ET\to\BT$ we obtain an isomorphism $\H^2(\BT,\ZZ) \iso \Hom(\Gamma,\ZZ)$.

\begin{theorem*}[{\cite[Thm.~10.7]{borel-hirzebruch-1958}}]
  Let $P\to P/G$ be a principal $G$-bundle, $\rho$ the projection $P/T \to P/H$, and $\eta$ the bundle of tangents along the fibers of the $G/H$ bundle $P/H \to P/G$. Then:
  \begin{align*}
    \rho^*(\mrm{p}(\eta))=\prod (1+\bar{a}_j^2)
  \end{align*}
  where $\{\pm \bar{a}_j\}_{1\le j\le k}$ are the roots of $G$ complementary to those of $H$, regarded as elements of $\H^2(P/T,\ZZ)$.
\end{theorem*}

Applied to the principal $G$-bundle $\mrm{E}G\to \BG$, for which $\rho=\Bi_{T,G}:\BT\to\BG$, this gives a formula for (the pullback to $\H^*(\BT,\ZZ)$ of) the characteristic class $\s_I(\eta)$ of the bundle $\eta$ of tangents along the fibers of the universal $G/H$ bundle $\BH\to \BG$, namely:
\begin{align*}
  \Bi_{T,G}^*\big( \s_I(\eta) \big) = \s_I\big(\bar{a}_1^2,\dots,\bar{a}_k^2\big) \; \in \; \H^*(\BT,\ZZ)
\end{align*}

\bigskip

This formula together with the following Lie-theoretic description of the pushforward:
\begin{align*}
  \Bi_{H,G*} : \H^*(\BH,\ZZ) \to \H^*(\BG,\ZZ)
\end{align*}
will enable us to prove Theorem~\ref{thm:ker-witten-genus}. To describe the latter, we need to introduce further notation (again, see \cite[Ch.~1]{borel-hirzebruch-1958} for more detail).

\smallskip

Fix a positive definite metric on $\mathfrak{g}$ invariant under the adjoint representation of $G$. It determines a metric on $V$ and hence a canonical isomorphism between $V$ and its dual space $V^*$ as well as a metric on $V^*$. A symmetry $\mrm{S}_a$ of $V$ with respect to a hyperplane $a=0$ induces a symmetry of $V^*$, also denoted $\mrm{S}_a$, defined by:
\begin{align*}
  \mrm{S}_a(b) = b-2(a,b)(a,a)^{-1} \cdot a
\end{align*}
The Weyl group $\W(G)$ of $G$ is the group of automorphisms of $T$ induced by inner automorphisms of $G$ leaving $T$ invariant. It may also be viewed as the group of isometries of $V$ leaving $\Gamma$ and the root diagram invariant. It is generated by the symmetries $\mrm{S}_{a_i}$ to the hyperplanes $a_i=0$ ($i=1,\dots,m$). The \emph{sign} of an element $w$ of $\W(G)$, denoted $\sgn(w)$, is the determinant of $w$ viewed as a linear transformation of $V$; it always equals $\pm1$. Choose a basis $e_1,\dots,e_l$ for $V^*$. Call a root $a=a_1e_1+\cdots+a_le_l$ \emph{positive} if the first nonvanishing coefficient $a_i$ is $>0$. Call a positive root \emph{simple} if it is not the sum of two positive nonzero roots. The simple roots form a basis for $V^*$ and every root is a linear combination, with integral coefficients of the same sign, of simple roots.

Let $\til{\mrm{e}}(G/T) \in \H^*(\BT,\ZZ)$ be the Euler class of the bundle of tangents along the fibers of $\BT\to\BG$. Up to sign it is the product of a set of positive roots of $G$, regarded as elements of $\H^*(\BT,\ZZ)$.  More precisely, it is the product of the roots of an invariant almost complex structure on $G/T$. Note that $G/T$ always admits a complex structure and that although the individual roots associated to an almost complex structure depend on the almost complex structure, their product does not. (See \cite[\S12.3, \S13.4]{borel-hirzebruch-1958}.)

\smallskip

The key to describing $\Bi_{H,G*}$ is the following:

\begin{theorem}[Borel-Hirzebruch, {\cite[Thm.~20.3]{borel-hirzebruch-1959}}]
  \label{thm:borel-hirzebruch}
  If $t \in \H^*(\BT,\ZZ)$ then:
  \begin{align*}
    \sum_{w \in \W(G)} \sgn(w) \cdot w(t) = \Bi_{T,G}^*\big(\Bi_{T,G*}(t)\big) \cdot \til{\mrm{e}}(G/T)
  \end{align*}
\end{theorem}

\begin{corollary}
  \label{BHToBG}
  \label{cosetpushforward}
  If $h \in \H^*(\BH,\ZZ)$ then:
  \begin{align*}
    \Bi_{T,G}^* \, \Bi_{H,G*}(h)
    = \sum_{[w] \in \W(G)/\W(H)} w\left(
      \frac{\til{\mrm{e}}(H/T)}{\til{\mrm{e}}(G/T)} \, \Bi_{T,H}^*(h) \right)
  \end{align*}
  where the sum runs over the cosets of $\W(H)$ in $\W(G)$.
\end{corollary}

(Note that this is a formula in the polynomial ring $\H^*(\BT,\ZZ)$.)

\begin{proof}
  Since $\Bi_{T,H*}\big(\til{\mrm{e}}(H/T)\big)=\chi(H/T)=|\W(H)| \in \H^0(\BH,\ZZ)$, write:
  \begin{align*}
    \Bi_{T,G}^* \, \Bi_{H,G*}(h) &= \Bi_{T,G}^* \, \Bi_{H,G*} \left(
      \frac{\Bi_{T,H*}\big(\til{\mrm{e}}(H/T)\big)}{|\W(H)|} \cdot h \right) \\
    \intertext{Apply the projection formula \cite[Prop.~8.2]{borel-hirzebruch-1958} to obtain:}
    \Bi_{T,G}^* \, \Bi_{H,G*}(h) &= \frac{1}{|\W(H)|} \, \Bi_{T,G}^* \, \Bi_{H,G*} \, \Bi_{T,H*} \left( \til{\mrm{e}}(H/T) \cdot \Bi_{T,H}^*(h) \right) \\
    &= \frac{1}{|\W(H)|} \, \Bi_{T,G}^* \, \Bi_{T,G*} \left(\til{\mrm{e}}(H/T) \cdot
      \Bi_{T,H}^*(h) \right) \\
    \intertext{Apply Theorem~\ref{thm:borel-hirzebruch} to obtain:}
    \Bi_{T,G}^* \, \Bi_{H,G*}(h) &= \frac{1}{|\W(H)|} \cdot \frac{1}{\til{\mrm{e}}(G/T)} \sum_{w \in \W(G)} \sgn(w) \cdot w\big(\til{\mrm{e}}(H/T) \cdot \Bi_{T,H}^*(h)\big) \\
    \intertext{Since $w\big(\til{\mrm{e}}(G/T)\big)=\sgn(w) \; \til{\mrm{e}}(G/T)$:}
    \Bi_{T,G}^* \, \Bi_{H,G*}(h) &= \frac{1}{|\W(H)|} \sum_{w \in \W(G)} w\left(
      \frac{\til{\mrm{e}}(H/T)}{\til{\mrm{e}}(G/T)} \, \Bi_{T,H}^*(h) \right)
    \intertext{Since $\W(G)$ acts on $\H^*(\BT,\ZZ)$ by ring homomorphisms, since if $w \in \W(H)$ then $w(\til{\mrm{e}}(H/T))=\sgn(w)\;\til{\mrm{e}}(H/T)$ and $w(\til{\mrm{e}}(G/T))=\sgn(w)\;\til{\mrm{e}}(G/T)$, and since $\Bi_{T,H}^*$ maps to the $\W(H)$-invariant subring of $\H^*(\BT,\ZZ)$, this sum may be written over the cosets of $\W(H)$ in $\W(G)$:}
    \Bi_{T,G}^* \, \Bi_{H,G*}(h) &= \sum_{[w] \in \W(G)/\W(H)} w\left(
      \frac{\til{\mrm{e}}(H/T)}{\til{\mrm{e}}(G/T)} \, \Bi_{T,H}^*(h) \right)
    \qedhere
  \end{align*}
\end{proof}

Now we specialize to the Cayley plane $G/H=\Ff/\Spin(9)$ (see \cite[\S19]{borel-hirzebruch-1958} and \cite[Plate~VIII]{bourbaki-1968} for more detail).

\smallskip

The extended Dynkin diagram of the root system $\Ff$ is:
\begin{align*}
  &\xymatrix@R1pt@C10pt@!C{ *{\bullet} \ar@{-}[r] & *{\circ} \ar@{-}[r] 
    & *{\circ} \ar@{=}[r] |*=0{>} & *{\circ} \ar@{-}[r] & *{\circ} \\
    -\til{a} \quad & a_1 & a_2 & a_3 & a_4 }
\intertext{A choice of simple roots is:
  \begin{align*}
    a_1 = e_2-e_3 \quad\quad\quad
    a_2=e_3-e_4 \quad\quad\quad
    a_3=e_4 \quad\quad\quad
    a_4=\tfrac12(e_1-e_2-e_3-e_4)
  \end{align*}
  Since the coefficient of $a_4$ in the maximal root:
  \begin{align*}
    \til{a}=2a_1+3a_2+4a_3+2a_4=e_1+e_2
  \end{align*}
  is prime, a theorem of Borel \& de Siebenthal \cite{borel-de-siebenthal-1949} implies that erasing $a_4$ from the extended Dynkin diagram gives the Dynkin diagram:}
&\xymatrix@R1pt@C10pt@!C{ *{\circ} \ar@{-}[r] & *{\circ} \ar@{-}[r] 
    & *{\circ} \ar@{=}[r] |*=0{>} & *{\circ} \\
    -\til{a}\quad & a_1 & a_2 & a_3 }
\end{align*}
of a subgroup of the compact Lie group $\Ff$. This type-$\mrm{B}_4$ subgroup is globally isomorphic to $\Spin(9)$, the 1-connected double cover of $\SO(9)$.

\medskip

The roots of this type-$\mrm{B}_4$ root subsystem are:
\begin{align*}
  \begin{cases}
    \pm e_i & 1 \le i \le 4 \\
    \pm e_i \pm e_j & 1 \le i < j \le 4
  \end{cases}
\end{align*}
The roots of $\Ff$ are these roots together with the complementary roots:
\begin{align*}
  \tfrac12 (\pm e_1 \pm e_2 \pm e_3 \pm e_4)
\end{align*}

Let $\mrm{T}$ be the standard maximal torus of $\SO(9)$ (see \cite[\S19.2]{borel-hirzebruch-1958}). Its preimage $\mrm{T}'$ under the double covering $\Spin(9)\to\SO(9)$ is a maximal torus of $\Spin(9)$ and hence also of $\Ff$.
The double covering $\mrm{T}'\to\mrm{T}$ determines an index-2 sublattice:
\begin{align*}
  \H^1(\mrm{T},\ZZ) \into\H^1(\mrm{T}',\ZZ)
\end{align*}
corresponding, under the identification $\H^1(T,\ZZ)\iso\Hom(\Gamma,\ZZ)$, to the index-2 sublattice:
\begin{align*}
  \ZZ\langle e_1,e_2,e_3,e_4 \rangle \into \ZZ\big\langle &\tfrac12(e_1+e_2+e_3+e_4), \tfrac12(e_1+e_2+e_3-e_4), \\& \tfrac12(e_1+e_2-e_3+e_4), \tfrac12(e_1-e_2+e_3+e_4) \big\rangle
\end{align*}

The following positive roots determine an almost complex structure on $\Spin(9)/\mrm{T}'$:
\begin{align*}
  \begin{cases}
    e_i & 1 \le i \le 4 \\
    e_i \pm e_j & 1 \le i < j \le 4
  \end{cases}
\end{align*}
These together with the following complementary positive roots determine an almost complex structure on $\Ff/\mrm{T}'$:
\begin{align*}
  \big\{\bar{a}_1,\dots,\bar{a}_8\big\} = \big\{ \tfrac12(e_1 \pm e_2 \pm e_3 \pm e_4) \big\}
\end{align*}
(The ordering of these roots will not matter.)

The 3~cosets of $\W(\Spin(9))$ in $\W(\Ff)$ are represented by the reflections:
\begin{align*}
  \big\{1, \; \mrm{S}_{a_4}, \; \mrm{S}_{a_4} \mrm{S}_{a_3} \mrm{S}_{a_4}\big\}
\end{align*}
which act, with respect to the basis $(e_1,\dots,e_4)$, by the matrices:
\begin{align*}
  \left\{
    \begin{pmatrix}
      1 & 0 & 0 & 0 \\
      0 & 1 & 0 & 0 \\
      0 & 0 & 1 & 0 \\
      0 & 0 & 0 & 1
    \end{pmatrix}, \;
    \tfrac12
    \begin{pmatrix}
      1 & \hphantom{-}1 & \hphantom{-}1 & \hphantom{-}1 \\
      1 & \hphantom{-}1 & -1 & -1 \\
      1 & -1 & \hphantom{-}1 & -1 \\
      1 & -1 & -1 & \hphantom{-}1
    \end{pmatrix}, \;
    \tfrac12
    \begin{pmatrix}
      \hphantom{-}1 & \hphantom{-}1 & \hphantom{-}1 & - 1 \\
      \hphantom{-}1 & \hphantom{-}1 & -1 & \hphantom{-}1 \\
      \hphantom{-}1 & -1 & \hphantom{-}1 & \hphantom{-}1 \\
      -1 & \hphantom{-}1 & \hphantom{-}1 & \hphantom{-}1
    \end{pmatrix}
  \right\}
\end{align*}
In particular they act on the set of positive complementary roots $\{\bar{a}_1,\dots,\bar{a}_8\}$ by:
\begin{align*}
  \{\bar{a}_1,\dots,\bar{a}_8\} = 
    \big\{ &\tfrac12(e_1 \pm e_2 \pm e_3 \pm e_4)\big\} \\
    \mrm{S}_{a_4}(\{\bar{a}_1\dots,\bar{a}_8\}) = \big\{ &e_1,e_2,e_3,e_4, \tfrac12(e_1+e_2+e_3-e_4),
    \tfrac12(e_1+e_2-e_3+e_4),\\& \tfrac12(e_1-e_2+e_3+e_4),
    \tfrac12(-e_1+e_2+e_3+e_4) \big\} \\
    \mrm{S}_{a_4}\mrm{S}_{a_3}\mrm{S}_{a_4}(\{\bar{a}_1,\dots,\bar{a}_8\}) = \big\{ &e_1,e_2,e_3,e_4,
    \tfrac12(e_1+e_2+e_3+e_4), \tfrac12(e_1+e_2-e_3-e_4),
    \\& \tfrac12(e_1-e_2+e_3-e_4),
    \tfrac12(-e_1+e_2+e_3-e_4) \big\}
\end{align*}

Thus:
\begin{proposition}
  \label{spin9f4pushforward}
  \begin{align*}
    \Bi_{\mrm{T}',\Ff}^* \Bi_{\Spin(9),\Ff*} \s_I(\eta) &=
    \frac{\s_I\big(\bar{a}_1^2,\dots,\bar{a}_8^2\big)}{\prod_i \bar{a}_i} + \mrm{S}_{a_4} \left(
      \frac{\s_I\big(\bar{a}_1^2,\dots,\bar{a}_8^2\big)}{\prod_i \bar{a}_i} \right) + \mrm{S}_{a_4} \mrm{S}_{a_3} \mrm{S}_{a_4}
    \left( \frac{\s_I\big(\bar{a}_1^2,\dots,\bar{a}_8^2\big)}{\prod_i \bar{a}_i} \right)
  \end{align*}
  where the complementary roots $\{\pm\bar{a}_1,\dots,\pm\bar{a}_8\} = \big\{\tfrac12(\pm e_1 \pm e_2 \pm e_3 \pm e_4)\big\}$ are regarded as elements of $\H^2(\mrm{BT}',\ZZ)$ and $\mrm{S}_{a_4},\mrm{S}_{a_4}\mrm{S}_{a_3}\mrm{S}_{a_4}$ act on them as described above.
\end{proposition}

\section{Proof of Theorem~\ref{thm:ker-witten-genus}}

The purpose of this section is to prove the following theorem, which was already stated in the introduction.

\theoremstyle{plain}
\newtheorem*{main-theorem}{Theorem~\ref{thm:ker-witten-genus}}
\begin{main-theorem}
  Away from~6, the ideal of $\pi_*\MString$ consisting of (bordism classes of) Cayley plane bundles with connected structure group is precisely the kernel of the Witten genus. In other words, the extension of this ideal in $\pi_*\MString\inv16$ is precisely the kernel of:
  \begin{align*}
    \phi_{\mathrm{W}} \tensor \ZZ\inv16 \;:\; \pi_*\MString\inv16 \;\to\; \pi_*\tmf\inv16 \iso \ZZ\inv16[\GG_4,\GG_6]
  \end{align*}
  where $\GG_4, \GG_6$ have degree $8,12$ respectively.
\end{main-theorem}

Since the Witten genus carries the subring of $\pi_*\MString\inv16$ generated by elements of degree $\le12$ isomorphically to the polynomial ring $\ZZ\inv16[\GG_4,\GG_6]$, and since (as discussed in the introduction) the Witten genus of any $\OP^2$ bundle with connected structure group vanishes, Theorem~\ref{thm:ker-witten-genus} can be proved by showing that $\OP^2$ bundles with connected structure group can serve as generators for $\pi_*\MString\inv16$ in dimensions $>12$.  And this can be done by constructing a set $S$ of such $\OP^2$ bundles which satisfy the conditions of Theorem~\ref{thm:mstring-gen} in all dimensions except $8$ and $12$.

\subsection*{Construction of \textit{M}$^\text{4\textit{n}}$}

The first step is to construct, for each $n\ge4$, a $\OP^2$ bundle $M^{4n}$ which satisfies condition~(1) of Theorem~\ref{thm:mstring-gen}. It will be a $\ZZ$-linear combination (topologically, a disjoint union with some string structures possibly reversed) of total spaces of $\OP^2$ bundles whose base spaces are products of two carefully chosen complete intersections.

Let $i:V^m(d_1,\dots,d_r) \into \CP^{m+r}$ denote a smooth complete intersection of degree $(d_1,\dots,d_r)$ and complex dimension~$m$. Consider the $\OP^2$ bundle pulling back from the universal bundle $\OP^2 \to \B\Spin(9) \to \B\Ff$ by a classifying map $g$ of the form:
\begin{align*}
  \hspace{-20pt}
  \xymatrix@C15pt{
    E \ar[rrrr]^-{\til{g}} \ar[d]_-\pi &&&& \B\Spin(9) \ar[d]^{\Bi_{\Spin(9),\Ff}} \\
    *\txt{$W\!\!=\!V^{m}({d_1,\dots,d_r})$ \\\quad\quad\quad\quad $\times
      V^{m'}({d'_1,\dots,d'_{r'}})$} \ar[r]^-{i \times i'}
    \ar@(d,dl)[rrrr]_g & \CP^{m+r} \times \CP^{m'+r'} \ar[r] &
    \CP^\infty \times \CP^\infty \ar[r] \ar@(ur,ul)[rr]^f & \mrm{BT}' \ar[r]
    & \B\Ff }
\end{align*}
where $m+m'=2n-8$.

Let $\H^*(\CP^\infty \times \CP^\infty) \iso \ZZ[x_1,x_2]$. Choose the map $f:\CP^\infty \times \CP^\infty \to \B\Ff$ so that $(e_1,e_2,e_3,e_4)$ pull back to $n_f \cdot (x_1,x_1,x_2,-x_2)$ respectively for some integer $n_f\ge1$. The generators $(e_2,e_3,e_4,\tfrac12(e_1-e_2-e_3-e_4))$ of the lattice $\H^2(\mrm{BT}',\ZZ)$ then pull back to $n_f \cdot (x_1,x_2,-x_2,0)$ respectively.

The degrees $(d_1,\dots,d_r)$ and $(d'_1,\dots,d'_{r'})$ need to be chosen so that $\p_1(\mrm{T}E)=0$ since this implies that $E$ admits a string structure. The exact sequences of vector bundles:
\begin{align*}
  \tag*{}\label{eqn:vb-ex-seq}
  \xymatrix@R1pt@C15pt{ 0 \ar[r]& \til{g}^*(\eta) \ar[r]& \mrm{T}E
    \ar[r]& \pi^* \mrm{T}\big(V^{m}({d_1,\dots,d_r})\times
      V^{m'}({d'_1,\dots,d'_{r'}})\big) \ar[r]& 0 \\
    0 \ar[r]& i^* \bigoplus_j \mrm{O}(d_j) \ar[r] &
    i^*\mrm{T}\CP^{m+r} \ar[r] & \mrm{T}V^m({d_1,\dots,d_r})
    \ar[r]& 0 }
\end{align*}
imply that:
\begin{align*}
  \p_1(\mrm{T}E) &= \tilde{g}^*\p_1(\eta) +
  \pi^*\p_1 \mrm{T}\big(V^{m}({d_1,\dots,d_r})\times V^{m'}({d'_1,\dots,d'_{r'}})\big) \\
  &= \tilde{g}^*\p_1(\eta) +
  \pi^*i^*\big[\p_1\mrm{T}\CP^{m+r} - \sum_j
  \p_1\mrm{O}(d_j)\big] + \pi^*i'^*\big[\p_1\mrm{T}\CP^{m'+r'} - \sum_{j'} \p_1\mrm{O}(d'_{j'})\big] \\
  &= \tilde{g}^*\p_1(\eta) + \pi^*(i\times i')^*\big[\big(
  m+r+1-\sum_j d_j^2 \big)x_1^2 + \big( m'+r'+1-\sum_{j'} (d'_{j'})^2
  \big)x_2^2\big]
  \intertext{The image of $\p_1(\eta)$ in $\H^4(\mrm{BT}')$ is $\sum
    \tfrac14(e_1\pm e_2 \pm e_3 \pm e_4)^2 =
    2(e_1^2+e_2^2+e_3^2+e_4^2)$, which pulls back to
    $4n_f(x_1^2+x_2^2)$. So:}
  \p_1(\mrm{T}E) &= \pi^* (i \times i')^* \Big[ (4n_f+m+1+r-\sum_j d_j^2)x_1^2 +
  (4n_f+m'+1+r'-\sum_{j'} (d'_{j'})^2)x_2^2 \Big]
\end{align*}

The following lemma shows that, for any given $m$ and $m'$, it is simple to choose degrees $(d_1,\dots,d_r)$ and $(d'_1,\dots,d'_{r'})$ so that this quantity vanishes, \emph{provided $n_f$ is sufficiently large.}  (The fact that the degrees can all be taken to be $2$'s and $3$'s is relevant since these are the primes inverted in this paper.)

\begin{lemma}
  \label{lemma:simple-sum-of-squares}
  For any integer $n\ge14$ there exist integers $a,b \ge 0$ so that:
  \begin{align*}
    n+(a+b) = a \cdot 2^2 + b \cdot 3^2
  \end{align*}
\end{lemma}
\begin{proof}
  This follows by induction since:
  \begin{align*}
    14+3 &= 2^2 + 2^2 + 3^2 &
    15+5 &= 2^2 + 2^2 + 2^2 + 2^2 + 2^2 &
    16+2 &= 3^2 + 3^2
  \end{align*}
  and since:
  \begin{align*}
    n+(a+b) &= a \cdot 2^2 + b \cdot 3^2  &&\implies&
    (n+3)+(a+1+b) &= (a+1) \cdot 2^2 + b \cdot 3^2 \qedhere
  \end{align*}
\end{proof}

As an aside, the values for $a$ and $b$ constructed in the proof are:
\begin{align*}
  a(n) &= 3n-8 \left\lceil n/3 \right\rceil &
  b(n) &= 3 \left\lceil n/3 \right\rceil -n
\end{align*}

\smallskip

Although the preceding lemma suffices to prove the results of this paper, the reader may find the reliance on complete intersections of arbitrarily high codimension unsatisfying. It is therefore worth noting that the following replacement for Lemma~\ref{lemma:simple-sum-of-squares} would make it possible to prove the results of this paper using complete intersections of codimension $\le4$.

\begin{conjecture}
  If $n\ge25$ then the GCD:
  \begin{align*}
    \mrm{GCD} \; \left\{ \prod_{i=1}^4 d_i \; \Big| \; 4n+4+1 =
      \sum_{i=1}^4 d_i^2,  \; d_i>0 \right\}
  \end{align*}
  has the form $2^a 3^b$ with $a+b>0$. In fact as $n$ increases from $25$, this $\mathrm{GCD}$ takes the values:
  \begin{align*}
    2^4 \cdot 3 \quad 2^3 \quad 2^4 \cdot 3^2 \quad 2^3 \cdot 3 \quad
    2^4 \quad 2^3 \cdot 3^2
  \end{align*}
  and then repeats from the beginning.
\end{conjecture}

We have to carefully choose the degrees $(d_1,\dots,d_r)$ and $(d'_1,\dots,d'_{r'})$ to ensure that the total space $E$ admits a string structure. However, these degrees have little effect on the Pontrjagin number $\s_n[E]$ which we compute next. Indeed, for dimension reasons:
\begin{align*}
  \s_n[E]=(i\times i')^*\;f^*\;\Bi_{\Spin(9),\Ff*}\;\s_n(\eta)
\end{align*}
Since the base space $W$ is a product of complete intersections, the pullback $(i \times i')^*x_1^mx_2^{m'}$ equals $(\prod_jd_j)(\prod_{j'}d'_{j'})$ times the fundamental class $[W]$. So the key is to compute the coefficients of the polynomial $f^*\Bi_{\Spin(9),\Ff*}\s_n(\eta)$ or, rather, their $\mathrm{GCD}$ as a function of $n$. This calculation lies at the heart of this paper. (It was the smoking gun which led to Theorem~\ref{thm:ker-witten-genus}.)

\begin{proposition}
  \label{prop:pushForwardBinom}
  \begin{align*}
    f^* \Bi_{\Spin(9),\Ff*} \s_n(\eta) = 2 n_f^{2n-8} \sum_{k=2}^{n-2}
    \left[ \binom{2n}{2} - \binom{2n}{2k} \right] x_1^{2k-4}
    x_2^{2n-2k-4}
  \end{align*}
\end{proposition}

\begin{proof}
  Since the polynomial in question is homogeneous in $n_f x_1$ and $n_f x_2$, we can, without loss of generality, simplify notation by setting $n_f=1$ and $(x_1,x_2)=(x,1)$.

  Proposition~\ref{spin9f4pushforward} gives the polynomial in the form of a power series:
  \begin{multline*}
    -\frac{1}{x^4}  \big(1+x^2+x^4+\cdots\big) \\
    \cdot \Biggl[\underbrace{-2 + (x+1)^{2n} + (x-1)^{2n}} -x^2 \big[
      \overbrace{-2+ (x+1)^{2n} + (x-1)^{2n} +2\binom{2n}2 } \big] \\
      +x^{2n}\big[2\binom{2n}2-2\big] +2x^{2n+2} \Biggr]
    \end{multline*}
    {The bracketed quantities differ by $2\dbinom{2n}2$ so the power series simplifies to the polynomial:}
    \begin{gather*}
      -\frac1{x^4} \cdot \Big[ -2 + (x+1)^{2n} + (x-1)^{2n}
      -2\binom{2n}2 (x^2+x^4+\cdots+x^{2n-2}) - 2x^{2n} \Big]
      \intertext{which simplifies further to:}
      2 \sum_{k=2}^{n-1} \left[ \binom{2n}2-\binom{2n}{2k} \right]
      x^{2k-4} \qedhere
  \end{gather*}
\end{proof}

\begin{proposition}
  \label{prop:gcddiffbinom}
  For any integer $n\ge4$ and any odd prime $p$:
  \begin{align*}
    \ord_p \left[ \mathop{\mrm{GCD}}_{1<k<n-1} \left\{ \binom{2n}{2} - \binom{2n}{2k} \right\} \right] =
    \begin{cases}
      1 & \text{if $2n=p^i-1$ or $2n=p^i+p^j$ for some $0 \le i \le j$} \\
      0 & \text{otherwise}
    \end{cases}
  \end{align*}
\end{proposition}

The key behind this is the following lemma.

\begin{lemma}
  \label{lemma:gcdbinom}
  For any integer $n>1$ and any odd prime $p$:
  \begin{align*}
    \ord_p \left[ \mathop{\mrm{GCD}}_{0<k<n} \binom{2n}{2k} \right] =
    \begin{cases}
      1 & \text{if $2n=p^i+p^j$ for some $0 \le i \le j$} \\
      0 & \text{otherwise}
    \end{cases}
  \end{align*}
\end{lemma}

It is worth comparing this result to the better known result that for any integer $n>1$ and any prime~$p$:
\begin{align*}
  \ord_p \left[ \mathop{\mrm{GCD}}_{0<k<n} \binom{n}{k} \right] =
  \begin{cases}
    1 & \text{if $n=p^i$ for some integer $i\ge0$} \\
    0 & \text{otherwise}
  \end{cases}
\end{align*}
Notice that, for any given integer $n>1$, at most one prime divides the latter $\mrm{GCD}$ whereas several primes may divide the former. For example, if $n=7$ then $2n=7^1+7^1=13^0+13^1$ and indeed $\mrm{GCD}_{0<k<7} \binom{14}{2k} = 7 \cdot 13$.

\begin{proof}[Proof of Lemma~\ref{lemma:gcdbinom}]
  By Kummer's theorem $\binom{2n}{2k}$ is divisible by $p$ if and only if there is at least 1 carry when adding $2k$ to $2n-2k$. Consider the base-$p$ expansion $\sum n_i p^i$ of an even integer $2n$. If there is a digit $n_i\ge2$ then there is no carry when adding $2p^i$ to $2n-2p^i$. If there are 2 distinct nonzero digits $n_i,n_j$ then there is no carry when adding $p^i+p^j$ to $2n-p^i-p^j$. If $2n=p^i+p^j$ and $0<2k<2n$ then there is always a carry when adding $2k$ to $2n-2k$, even if $i=j$. These 3 facts together imply the first part of the lemma. The second part of the lemma follows from the fact that if $j>0$ then there is precisely 1 carry when adding $(p-1)p^{j-1}$ to $p^i+p^j-(p-1)p^{j-1}$. (If $j=0$ then the second part of the lemma is vacuous.)
\end{proof}

\begin{proof}[Proof of Proposition~\ref{prop:gcddiffbinom}]
  If an odd prime $p$ divides the $\mathrm{GCD}$ then all the binomial coefficients $\binom{2n}{2k}$ for $0<2k<2n$ must be congruent mod~$p$. If they are all congruent to $0$ mod~$p$ then Lemma~\ref{lemma:gcdbinom} applies and $2n=p^i+p^j$ for some $0\le i \le j$. So suppose that the binomial coefficients are all nonzero mod~$p$. By Kummer's theorem this happens precisely when for each $0<2k<2n$ there are no carries when adding $2k$ to $2n-2k$. This in turn happens precisely when $2n=l\cdot p^i-1$ for some $i>0$ and some (odd) $0<l<p$. According to Lucas's theorem (see \cite[\S1]{granville97}), if $l>1$ then:
  \begin{align*}
    \binom{l\cdot p^i-1}{p^i+1}
    &\equiv \binom{p-1}1 \binom{p-1}0 \cdots \binom{p-1}0
    \binom{l-1}1 \equiv 1-l \mod p
  \end{align*}
  However:
  \begin{align*}
    \binom{l\cdot p^i-1}2 \equiv 1 \mod p
  \end{align*}
  So all the binomial coefficients can be congruent mod~$p$ only if $l=1$, and indeed the congruence $(1+x)^{p^i} \equiv 1+x^{p^i}$ mod~$p$ implies that:
  \begin{align*}
    (1+x)^{p^i-1} \equiv (1+x^{p^i})(1+x)^{-1} =
    1-x+x^2-x^3+\cdots+x^{p^i-1} \mod p
  \end{align*}
  and hence that:
  \begin{align*}
    \binom{p^i-1}{2k} \equiv 1 \mod p
  \end{align*}
  for all $0<2k<p^i-1$. 

  \medskip

  It remains to show that the $\mathrm{GCD}$ is never divisible by $p^2$ for $p$ odd. By the preceding argument it remains only to show this when $2n=p^i+p^j$ or $2n=p^i-1$ for $0\le i \le j$. Remember that by assumption $2n\ge16$.

  Suppose first that $2n=p^i+p^j$. If $i>1$ then there are at least 2 carries when adding $2$ to $p^i+p^j-2$; so by Kummer's theorem $\binom{2n}2$ is congruent to $0$ mod~$p^2$ while by Lemma~\ref{lemma:gcdbinom} $\binom{2n}{2k}$ is nonzero mod~$p^2$ for some $0<2k<2n$.
If $i\le1$ then since $1+1<16$ we may assume that $j\ge1$ and split into 2 cases: $p^j+1$ and $p^j+p$. When $j\ge2$ the 1st case can be handled as when $i>1$. The remaining cases are handled by the following straightforward congruences mod~$p^2$:
  \begin{align*}
    \binom{p+1}2-\binom{p+1}4 &\equiv \tfrac{5}{12}p &
    \binom{p^j+p}2-\binom{p^j+p}4 &\equiv -\tfrac14(p^j+p)
  \end{align*}
  The coefficient $\tfrac5{12}$ is not a problem since $2n=p+1\ge16$ only if $p\ge17$.

  Suppose now that $2n=p^i-1$. Consider the following congruences mod~$p^2$:
  \begin{align*}
    \binom{p^i-1}2 &\equiv 1 - \tfrac32 p^i &
    \binom{p-1}{4} &\equiv 1-\tfrac{25}{12}p &
    \binom{p^i-1}{p^{i-1}+p^{i-2}} &\equiv 1-p
  \end{align*}
  The 1st and 2nd are immediate, and subtracting them gives the desired result for $i=1$. (The resulting coefficient $-\tfrac32+\tfrac{25}{12}=\tfrac7{12}$ of $p$ is not a problem since $2n=p-1\ge16$ only if $p\ge17$.) Subtracting the 3rd congruence from the 1st gives the desired result when $i\ge2$ but proving the 3rd congruence is more subtle. Here, and quite often in what follows, we rely on the following powerful theorem.

  \theoremstyle{plain}
  \newtheorem*{gran-thm}{Granville's theorem}
  \begin{gran-thm}[{\cite[Thm.~1]{granville97}}]
    Suppose that a prime power $p^q$ and positive integers $n=m+r$ are given. Write $n=n_0+n_1p+\cdots+n_dp^d$ in base $p$, and let $N_j$ be the least positive residue of $[n/p^j]$ mod~$p^q$ for each $j\ge0$ (so that $N_j=n_j+n_{j+1}p+\cdots+n_{j+q-1}p^{q-1}$); also make the corresponding definitions for $m_j,M_j,r_j,R_j$. Let $e_j$ be the number of indices $i\ge j$ for which $n_i<m_i$ (that is, the number of `carries', when adding $m$ and $r$ in base $p$, on or beyond the $j$th digit). Then:
    \begin{align*}
      \frac1{p^{e_0}} \equiv (\pm1)^{e_{q-1}}
      \left(\frac{(N_0!)_p}{(M_0!)_p (R_0!)_p}\right)
      \left(\frac{(N_1!)_p}{(M_1!)_p (R_1!)_p}\right)
      \cdots
      \left(\frac{(N_d!)_p}{(M_d!)_p (R_d!)_p}\right)
      \mod p^q
    \end{align*}
    where $(\pm1)$ is $(-1)$ except if $p=2$ and $q\ge3$. Here $(n!)_p$ denotes the product of those integers $\le n$ which are not divisible by $p$.
  \end{gran-thm}

  We need to show that the 3rd congruence holds for $i\ge2$ but assume first that $i\ge3$. Then according to Granville's theorem the binomial coefficient $\binom{p^i-1}{p^{i-1}+p^{i-2}}$ is congruent to:
  \begin{gather*}
    \frac{((p^2-1)!)_p}{(p!)_p \cdot ((p^2-p-1)!)_p}
    \cdot
    \frac{((p^2-1)!)_p}{((p+1)!)_p \cdot ((p^2-p-2)!)_p}
    \cdot
    \frac{((p-1)!)_p}{(1!)_p \cdot ((p-2)!)_p} \mod p^2
  \end{gather*}
    {Gathering common factors gives:}
    \begin{align*}
    \binom{p^i-1}{p^{i-1}+p^{i-2}} &\equiv \left(\frac{(1-p)(2-p)\cdots((p-1)-p))}{(p!)_p}\right)^2
    \cdot \frac{p^2-p-1}{p+1} \cdot (p-1) \mod p^2
    \\
    &\equiv
    \Big(1-\underbrace{p(1+\tfrac12+\tfrac13+\cdots+\tfrac1{p-1})}_{\equiv0}\Big)^2
    \cdot (1-p) \mod p^2
    \end{align*}
    {The bracketed quantity is congruent to $0$ mod~$p^2$ since by Wolstenholme's theorem \cite[Thm.~116]{hardy-wright-1979} (``Wolstenholme\dots he was despondent and dissatisfied and consoled himself with mathematics and opium''---Sir Leslie Stephen, Virginia Woolf's father) the $(p-1)$st harmonic number is congruent to $0$ mod~$p^2$ for $p>3$ and to $2p$ for $p=3$. Thus we obtain:}
    \begin{align*}
      \binom{p^i-1}{p^{i-1}+p^{i-2}} \equiv 1^2 \cdot (1-p) = 1-p \mod p^2
    \end{align*}
  If $i=2$ then the first factor in the congruence provided by Granville's theorem disappears, and the square in the following congruences therefore does too but, since $1^2=1$, this does not affect the final result.
\end{proof}

\subsection*{Construction of \textit{N}$^{\text{2(\textit{p}$^\textit{i}$+\textit{p}$^\textit{j}$)}}$}

\newcommand{\pip}{\ensuremath{{(p^i+1)/2}}}
\newcommand{\pjm}{\ensuremath{{(p^j-1)/2}}}
\newcommand{\tpip}{\ensuremath{{p^i+1}}}
\newcommand{\tpjm}{\ensuremath{{p^j-1}}}

The second step is to construct, for each prime $p>3$ and $0<i<j$, a $\OP^2$ bundle $N^{2(p^i+p^j)}$ which satisfies condition~(2) of Theorem~\ref{thm:mstring-gen}. 
Throughout this section let $p>3$ and $0<i<j$ be arbitrary but fixed and, to simplify notation, let:
\begin{align*}
  n&=\tfrac12(p^j-1) &
  n'&=\tfrac12(p^i+1)
\end{align*}

The goal then is to construct a $\OP^2$ bundle $N^{4(n+n')}$ with:
\begin{align*}
  \s_{n+n'}[N^{4(n+n')}]&=0\\
  \s_{n,n'}[N^{4(n+n')}] &\not\equiv 0 \mod p^2
\end{align*}
To do this, we will construct two $\OP^2$ bundles $E_1$ and $E_2$ and define:
\begin{align*}
  N^{4(n+n')} = \mrm{LCM}\big(\s_{n+n'}[E_1],\s_{n+n'}[E_2]\big) \cdot
  \left( \frac{E_1}{\s_{n+n'}[E_1]} - \frac{E_2}{\s_{n+n'}[E_2]}
  \right)
\end{align*}
Then $\s_{n+n'}[N^{4(n+n')}]=0$, so all that will remain will be to show that $\s_{n,n'}[N^{4(n+n')}] \not\equiv 0$ mod~$p^2$. To do so, it will suffice to show that:
\begin{align*}
  \s_{n,n'}[E_1] &\equiv 0 \mod p^2 \\
  \s_{n,n'}[E_2] &\not\equiv 0 \mod p^2 \\
  \ord_p \; \s_{n+n'}[E_1] &\le \ord_p \;
  \s_{n+n'}[E_2]
\end{align*}

\medskip

Above we saw that the characteristic number $\s_n[E]$ depends only on the image of $\s_n(\eta)$ in $\H^*(E)$ and \emph{not} on the Pontrjagin classes of the base $W$. The characteristic number $\s_{n,n'}[E]$ is more subtle, however. Indeed, for a bundle $\OP^2 \to E \xto{\pi} W$ classified as before by a map $g : W \to \B\Ff$,
we have:
\begin{align*}
  \s_{n,n'}(\mrm{T}E) \;=\;
  & \til{g}^*\s_{n,n'}(\eta) \\
  &+ \pi^*\s_{n}(\mrm{T}W) \cdot \til{g}^*\s_{n'}(\eta) \\
  &+ \pi^*\s_{n'}(\mrm{T}W) \cdot
  \til{g}^*\s_{n}(\eta) \\&+
  \pi^*\s_{n,n'}(\mrm{T}W) 
  \intertext{Applying the $\H^*(W)$-module homomorphism $\B\pi_*$ (which decreases degrees by 16) gives:}
  \B\pi_* \s_{n,n'}(\mrm{T}E) \;=\;
  & g^*\Bi_*\s_{n,n'}(\eta) \\
  &+ \s_{n}(\mrm{T}W) \cdot g^*\Bi_*\s_{n'}(\eta) \\
  &+ \s_{n'}(\mrm{T}W) \cdot
  g^*\Bi_*\s_{n}(\eta)
\end{align*}
To compute the last two terms, note that the 2nd exact sequence of vector bundles on p.~\pageref{eqn:vb-ex-seq} implies that:
\begin{align*}
  \s_{n}(\mrm{T}W) &= \s_n\big(\mrm{T}V^m(d_1,\dots,d_r) \times
  \mrm{T}V^{m'}(d'_1,\dots,d'_{r'})\big) \\
  &= i^* \Big( \s_{n}(\CP^{m+r}) - \sum_j \s_{n} \mrm{O}(d_j) \Big) +
  i'^* \Big( \s_{n}(\CP^{m'+r'}) - \sum_{j'} \s_{n} \mrm{O}(d'_{j'}) \Big) \\
  &= (i \times i')^* \Big[ \Big(m+r+1-\sum_j d_j^{2n} \Big) x_1^{2n} +
  \Big(m'+r'+1-\sum_{j'} (d'_{j'})^{2n} \Big) x_2^{2n} \Big]
\end{align*}

\medskip

Let $E_1$ be the $\OP^2$ bundle obtained by taking:
\begin{align*}
  (m,m')=(2n-2,2n'-6)=(p^j-3,p^i-5)
\end{align*}
in the construction of $E$ above. Then for dimension reasons:
\begin{align*}
  \s_{n}(\mrm TV^m)=\s_{n}(\mrm TV^{m'})=\s_{n'}(\mrm TV^{m'})=0
\end{align*}
and by Proposition~\ref{prop:pushForwardBinom}:
\begin{multline*}
  \B\pi_*\s_{n,n'}(\mrm TE_1) = \\
  g^* \Bi_*\s_{n,n'}(\eta)
  + \Big(m+r+1-\sum_j d_j^{2n'}\Big) \cdot \left[
    \binom{p^j-1}2-\binom{p^j-1}{p^j-p^i}\right] \cdot (i \times i')^*
  x_1^m x_2^{m'}
\end{multline*}
  {Part (1) of Corollary~\ref{cor:sn1n2eta} below shows that $g^* \Bi_*\s_{n,n'}(\eta) \equiv 0$ mod $p^2$ and Granville's theorem can be used to show that both binomial coefficients are congruent to $1$ mod~$p^2$ so:}
  \begin{align*}
    \B\pi_*\s_{n,n'}[E_1] \equiv\; 0 \mod p^2
  \end{align*}

\smallskip

Let $E_2$ be the $\OP^2$ bundle obtained by taking:
\begin{align*}
  (m,m')=(p^{j-1}-3,p^j-p^{j-1}+p^i-5)
\end{align*}
in the construction of $E$ above. Then for dimension reasons $\s_{n}(\mrm TV^m)=\s_{n}(\mrm TV^{m'})=0$. If $i=j-1$ then $\s_{n'}(\mrm TV^m)=0$ as well. So by Proposition~\ref{prop:pushForwardBinom}:
\begin{multline*}
  \B\pi_*\s_{n,n'}(\mrm TE_2) = \\
  \begin{aligned}
    g^* & \Bi_*\s_{n,n'}(\eta) \\    
  &+ \Big(m+r+1-\sum_j d_j^{p^i+1}\Big) \cdot \left[
    \binom{p^j-1}2-\binom{p^j-1}{p^{j-1}-p^i} \right] \cdot
  (1-\delta_{i=j-1}) \\
  &+ \Big(m'+r'+1-\sum_{j'} (d'_{j'})^{p^i+1}\Big) \cdot \left[
    \binom{p^j-1}2-\binom{p^j-1}{p^{j-1}+1}\right] \cdot (i \times
  i')^* x_1^m x_2^{m'}
  \end{aligned}
\end{multline*}
{(Here $\delta_P$ equals $1$ if $P$ is true and equals $0$ otherwise.)  Granville's theorem can be used to show that the first three binomial coefficients are congruent to $1$ mod~$p^2$ while the last is congruent to $1-p$ mod~$p^2$ so:}
\begin{gather*}
  \B\pi_*\s_{n,n'}(\mrm TE_2) \equiv\; g^* \Bi_*\s_{n,n'}(\eta)
  + \Big(m'+r'+1-\sum_{j'} (d'_{j'})^{p^i+1}\Big) \cdot p \cdot (i
  \times i')^* x_1^m x_2^{m'} \mod p^2
  \intertext{By Fermat's little theorem:}
  \Big(m'+r'+1-\sum_{j'} (d'_{j'})^{p^i+1}\Big) \equiv \Big(m'+r'+1-\sum_{j'} (d'_{j'})^2\Big) \mod p
\end{gather*}
  {Recall that the degrees $(d'_1,\dots,d'_{r'})$ are chosen (say using Lemma~\ref{lemma:simple-sum-of-squares}) to make the latter quantity equal $-4n_f$ (since this makes $\p_1(\mrm{T}E_2)=0$). So the particular degrees chosen are irrelevant here and:}
  \begin{align*}
  \B\pi_*\s_{n,n'}(\mrm TE_2) &\equiv\; g^* \Bi_*\s_{n,n'}(\eta) -4n_f \cdot p \cdot (i \times
  i')^* x_1^m x_2^{m'} \mod p^2 \\
  \intertext{By Part (2) of Corollary~\ref{cor:sn1n2eta} below $g^* \Bi_*\s_{n,n'}(\eta) \equiv 8p\cdot n_f^{m+m'} \cdot (i \times i')^* x_1^m x_2^{m'}$ mod $p^2$ so:}
  \B\pi_*\s_{n,n'}(\mrm TE_2) &\equiv\; (8n_f^{p^i+p^j-8}-4n_f) \cdot p \cdot (i \times i')^*
  x_1^m x_2^{m'} \mod p^2
  \intertext{By Fermat's little theorem:}
  \B\pi_*\s_{n,n'}(\mrm TE_2) &\equiv\; 4n_f(2n_f^{-7}-1) \cdot p \cdot (i \times i')^* x_1^m
  x_2^{m'} \mod p^2 \\
  \intertext{Since $W$ is a product of complete intersections, $(i \times i')^*x_1^mx_2^{m'}$ equals $(\prod_j d_j) (\prod_{j'} d'_{j'})$ times the fundamental class $[W]$, and the degrees are all chosen to be nonzero mod~$p$. Determining the roots of the polynomial $n_f^7-2$ mod~$p$ is a delicate task, but certainly if $n_f \equiv 1$ mod~$p$ then:}
  \B\pi_*\s_{n,n'}[E_2] &\not\equiv\; 0 \mod p^2
\end{align*}

\begin{lemma}
  \begin{align*}
    \ord_p \; \s_{n+n'}[E_1] \le \ord_p \; \s_{n+n'}[E_2]
  \end{align*}
\end{lemma}
\begin{proof}
  Assuming as we did above that $n_f \equiv 1$ mod~$p$, it suffices by Proposition~\ref{prop:pushForwardBinom} to show that:
  \begin{align*}
    \ord_p \left[\binom{p^i+p^j}2-\binom{p^i+p^j}{p^{j-1}+1}\right]
    \le \ord_p \left[ \binom{p^i+p^j}2-\binom{p^i+p^j}{p^j+3}\right]
  \end{align*}
  By Kummer's Theorem:
  \begin{align*}
    \ord_p \binom{p^i+p^j}2 &= i &
    \ord_p \binom{p^i+p^j}{p^{j-1}+1} &= i+1 &
    \ord_p \binom{p^i+p^j}{p^j+3} &= i
  \end{align*}
  So the difference of the 1st and 2nd binomial coefficients has order $i$ while the difference of the 1st and 3rd binomial coefficients has order $\ge i$ (in fact it has order $i+2$, as can be shown using Granville's theorem).
\end{proof}

\newcommand{\m}{n}
\newcommand{\n}{n'}

The method used to prove Proposition~\ref{prop:pushForwardBinom} can
be used to establish the following formula (which holds for any integers $\m>\n$, not just the integers we are concerned with here).

\begin{proposition}
  \label{prop:smnformula}
  \begin{multline*}
    f^* \Bi_* \s_{\m,\n}(\eta) =\\ 
    \begin{aligned}
      -4 n_f^{\m+\n-8} \sum_{k=2}^{\m+\n-1} \Biggl[
      \binom{2\m}{2k} &+ \binom{2\n}{2k} +\binom{2\n}{2k-2\m} +
      \binom{2\m}{2k-2\n} \\
      &+ \tfrac12 \sum_{l=0}^k (-1)^l \binom{2\n}{l} \binom{2\m-2\n}{2k-2l} \\
      &- \binom{2\m}2 \sum_{l=1}^{\m-1} \binom{2\n}{2k-2l}
      -\binom{2\n}2\sum_{l=1}^{\n-1}\binom{2\m}{2k-2l} \\
      & -\binom{2\n}2(1-\delta_{\n \le k \le
        \m}) -\binom{2\m}2(1+\delta_{\n+1 \le k \le \m-1}) \\
      &+ \tfrac12\binom{2\m+2\n}2 - 3 \delta_{k \in \{\m,\n\}} \Biggr]
      x_1^{2k-4} x_2^{2\m+2\n-2k-4}
    \end{aligned}
  \end{multline*}
  where $\delta_P$ equals $1$ if $P$ is true and equals $0$ otherwise.
\end{proposition}

\begin{corollary}
  \label{cor:sn1n2eta} \mbox{}
  \begin{enumerate}
  \item If $(m,m')=(2n-2,2n'-6)=(p^j-3,p^i-5)$ then the coefficient of $x_1^mx_2^{m'}$ in $f^*\Bi_*\s_{n,n'}(\eta)$ is congruent to $0$ mod~$p^2$.
  \item If $(m,m')=(p^{j-1}-3,p^j-p^{j-1}+p^i-5)$ then the coefficient of $x_1^mx_2^{m'}$ in $f^*\Bi_*\s_{n,n'}(\eta)$ is congruent to $8p \cdot n_f^{m+m'}$ mod~$p^2$.
  \end{enumerate}
\end{corollary}
\begin{proof}[Proof of part (1) of Corollary \ref{cor:sn1n2eta}]
  If $(m,m')=(2n-2,2n'-6)$ then the coefficient of $x_1^mx_2^{m'}$ is the $k=\m+1$ summand in Proposition~\ref{prop:smnformula}. It is not difficult to show that this summand is congruent mod~$p^2$ to:
  \begin{align*}
    4n_f^{(p^j+p^i)/2-8}
    \Big[0 &+ 0 + \tfrac12p^i + 1 \\
    &+ A \\
    &- (2^{p^i}-1-\tfrac12p^i) + \tfrac14p^i \\
    &-\tfrac12p^i -1 \\
    &-\tfrac14p^i-0\Big]
  \end{align*}
  where:
  \begin{align*}
    A = \tfrac12 \!\! \sum_{l=0}^{(p^j+1)/2} \!\!\! (-1)^l \binom{p^i+1}l \binom{p^j-p^i-2}{p^j-2l+1}
  \end{align*}
  Due to tidy pairwise cancellations, \emph{all that remains is to show that $A \equiv 2^p-1-\tfrac12p^i$ mod~$p^2$.} (Note that $n^{p^2-p} \equiv 1$ mod~$p^2$ for any integer $n \not\equiv 0$ mod~$p$ since the multiplicative group $(\ZZ/p^2)^\times$ has order $p^2-p$; it follows by induction that $n^{p^i} \equiv n^p$ mod~$p^2$ for any $i>0$.)

  \medskip

  (a) If $i>1$ then Granville's theorem can be used to show that:
  \begin{align*}
    A &\equiv \!\! \sum_{r=0}^{(p-1)/2} \!\!\! (-1)^r \binom{p}{r}
    \mod p^2 \\
    \intertext{(The key is that:
      \begin{align*}
        \binom{p^i+1}{l} &\equiv
        \left\{
          \begin{array}{cl}
            \dbinom{p+1}{l} & \text{if $i=1$} \\
            \dbinom{p}{r} & \parbox{\widthof{$l=rp^{i-1}+1$ with $0\le r\le p$}}{if $i>1$ and
              $l=rp^{i-1}$ or \\ $l=rp^{i-1}+1$ with $0\le r\le p$}  \\
            0 & \text{otherwise}
          \end{array}
        \right.
      \end{align*}
      mod~$p^2$.)
    }
    \intertext{By the identity $\sum_{j=0}^k(-1)^j\binom{n}{j}=(-1)^k\binom{n-1}k$ (proved inductively using Pascal's rule):}
    A&\equiv (-1)^{(p-1)/2} \binom{p-1}{(p-1)/2} \mod p^2 \\
    \intertext{By the eponymous congruence of Morley's ingenious 1895 paper \cite{morley-1895}:}
    A &\equiv 2^{2(p-1)} \mod p^2 \\
    \intertext{The final step is to show that $2^{2(p-1)} \equiv 2^p-1$ mod~$p^2$. Write:}
    2^{2(p-1)} &= (2^{p-1}+1)(2^{p-1}-1)+1 \\
    \intertext{By Fermat's little theorem the two factors are congruent to $2$ and $0$ mod~$p$ respectively, so:}
    A &\equiv 2 (2^{p-1}-1)+1 \mod p^2 \\
    &= 2^p-1
  \end{align*}

  \medskip

  (b) If $i=1$ then Granville's theorem can be used to show that:
  \begin{align*}
    A \; &\equiv p \;+\; \tfrac12 \!\! \sum_{l=0}^{(p-1)/2} \!\!\! (-1)^l
    \binom{p+1}l \binom{p-2}{2l-1} \mod p^2 \\
    \intertext{Since the 1st binomial coefficient is congruent to 0 mod~$p$ for $1<l<p$, we can simplify the 2nd binomial coefficient mod~$p$ via the congruence:
      \begin{align*}
        (1+x)^{p-2} &\equiv (1+x^p)(1+x)^{-2} = (1+x^p) \sum_{k=0}^\infty (-1)^k (k+1) x^k \mod p
      \end{align*}
      and, subtracting a correction factor, obtain:
    }
    A &\equiv \tfrac12p \;-\!\! \sum_{l=0}^{(p-1)/2} \!\!\! (-1)^l \binom{p+1}l \cdot l \mod p^2 \\
    \intertext{By the identity $\sum_{j=0}^k(-1)^j\binom{n}{j} j = (-1)^k \binom{n-2}{k-1} n$ (proved by writing $\binom{n}{j}=\binom{n-1}{j-1}\frac{n}{j}$ and then applying the earlier cited identity $\sum_{j=0}^k(-1)^j\binom{n}{j}=(-1)^k\binom{n-1}{k}$) and by the identity $\binom{n-2}{k-1} = \binom{n-2}{k}\frac{k}{n-k-1}$:}
    A &\equiv \tfrac12p - (-1)^{(p-1)/2} \binom{p-1}{(p-1)/2} \cdot \frac{p^2-1}{p+1} \mod p^2 \\
    \intertext{By Morley's congruence:}
    A &\equiv \tfrac12p + 2^{2(p-1)} (1-p) \mod p^2 \\
    \intertext{And again since $2^{2(p-1)}\equiv2^p-1$ mod~$p^2$:}
    A &\equiv 2^p-1-\tfrac12p \mod p^2 \qedhere
  \end{align*}
\end{proof}

\begin{proof}[Proof of part (2) of Corollary \ref{cor:sn1n2eta}]
  If $(m,m')=(p^{j-1}-3,p^j-p^{j-1}+p^i-5)$ then the coefficient of $x_1^mx_2^{m'}$ is the $k=\tfrac12(p^{j-1}+1)$ summand in Proposition~\ref{prop:smnformula}. It is not difficult to show that this summand is congruent mod~$p^2$ to:
  \begin{align*}
      4n_f^{(p^j+p^i)/2-8}
      \big[ (1-p) &+ \delta_{i=j-1} + 0 + 1 \\
      &+ B \\
      &- (2^p-\delta_{i=j-1}) - (-\tfrac14p^i) \\
      &-0 -(2-\delta_{i=j-1}) \\
      &-\tfrac14p^i -3\delta_{i=j-1} \big]
  \end{align*}
  where:
  \begin{align*}
    B &= \tfrac12 \!\!\! \sum_{l=0}^{(p^{j-1}+1)/2} \!\!\!\!\!\! (-1)^l
    \binom{p^i+1}l \binom{p^j-p^i-2}{p^{j-1}-2l+1}
  \end{align*}
  Due to tidy cancellations, \emph{all that remains is to show that $B \equiv 2^p-p$ mod~$p^2$.}

  \medskip

  (a) If $i>1$ then the above stated fact about $\binom{p^i+1}l$ can be used to show that:
  \begin{gather*}
    B \equiv \tfrac12 \sum_{r=0}^p (-1)^r \binom{p}{r} \binom{p^j-p^i-2}{p^{j-1}-2rp^{i-1}+1} \mod p^2
    \intertext{The 1st binomial coefficient is congruent to $0$ mod~$p$ for $0<r<p$. The 2nd binomial coefficient is congruent to $0$ mod~$p$ if $0<r<\tfrac12(p+1)$ and congruent to $-2$ mod~$p$ if $\tfrac12(p+1)\le r<p$. So:}
    B \equiv \tfrac12 \binom{p^j-p^i-2}{p^{j-1}+1} 
      \quad-\!\!\!{\sum_{r=(p+1)/2}^{p-1} \!\!\! (-1)^r \binom{p}{r}}
      \quad-\tfrac12\binom{p^j-p^i-2}{p^{j-1}-2p^i+1} \mod p^2
    \end{gather*}
    {Granville's \& Wolstenholme's theorems can be used to simplify the first and last terms mod~$p^2$ while the identity $\sum_{j=0}^k (-1)^j\binom{n}{j} = (-1)^k\binom{n-1}k$ can be used to simplify the summation, yielding:
      \begin{multline*}
        B \equiv (1+\delta_{i=j-1}-p\delta_{i\ne j-2})
        \quad {-1+(-1)^{(p-1)/2}\binom{p-1}{(p-1)/2}} \quad
        -(\delta_{i=j-1}+p\delta_{i=j-2}-1) \\
        \mod p^2
      \end{multline*}
      By Morley's congruence:}
    \begin{align*}
      B &\equiv 2^{2(p-1)}+1-p \mod p^2
      \intertext{And again since $2^{2(p-1)} \equiv 2^p-1$ mod~$p^2$:}
      B &\equiv 2^p-p \mod p^2
    \end{align*}

    (b) If $i=1$ then:
    \begin{align*}
      B &= \tfrac12 \sum_{l=0}^{p+1} (-1)^l \binom{p+1}l
      \binom{p^j-p-2}{p^{j-1}-2l+1}
      \intertext{Granville's theorem can be used to show that:}
      B&\equiv \tfrac12(p+2) + \tfrac12 \!\!\! \sum_{l=0}^{(p-1)/2}
      \!\!\!  (-1)^l \binom{p+1}l \binom{p-2}{2l-1} \mod p^2
      \intertext{This summation appeared above in the \textsc{Proof of part (1) of Corollary~\ref{cor:sn1n2eta}}, part (b). In fact $B \equiv A+1-\tfrac12p$ mod~$p^2$. Since we concluded that $A \equiv 2^p-1-\tfrac12p^i$ mod $p^2$, it follows that:}
    &B\equiv 2^p-p \mod p^2 \qedhere
    \end{align*}
\end{proof}

\section*{Acknowledgments}

Thanks to Burt Totaro for his inspiration \& support. Thanks to Neil Strickland for pointing out in the early stages of this work the encouraging fact that numbers of the form $p^i+p^j$ play a special role in the complex oriented cohomology of $\BU\bb6$. Thanks to Jack Morava, Nitu Kitchloo, Bill Browder, Bernhard Hanke, Baptiste Calm\`es, Stephanie Belcher, Volker Genz, Pankaj Vishe \& Michael Ontiveros for helpful discussions. Thanks to Haynes Miller and the referee for helpful suggestions. Thanks to the Cambridge Philosophical Society, the Cambridge Lundgren Fund \& the Max Planck Institute for Mathematics in Bonn for their generous financial support. Thanks also to Cassius Garten for providing a vegetarian oasis in a vast culinary landscape of \emph{Wurst \& Sauerbraten}. Last, but certainly not least, thanks to Jessie Codner for her infectious creativity.

\def\cprime{$'$}

\end{document}